\documentclass[a4paper,11pt]{article}

\usepackage{bm}
\usepackage{dsfont}
\usepackage{amsmath}
\usepackage{amsfonts}
\usepackage{amssymb}
\usepackage{graphicx}
\usepackage{url}
\usepackage{geometry}
\newtheorem{set2}{Satz}[section]
\newtheorem{theorem}[set2]{Theorem}

\newtheorem{corollary}[set2]{Corollary}

\newtheorem{definition}[set2]{Definition}

\newtheorem{lemma}[set2]{Lemma}
\newtheorem{notation[set2]}{Notation}

\newtheorem{proposition}[set2]{Proposition}
\newtheorem{remark}[set2]{Remark}

\newcommand{\ep}{\hfill{$\square$}}
\newenvironment{proof}[1][Proof]{\textit{#1.} }{\
\\}

\def\XXint#1#2#3{{\setbox0=\hbox{$#1{#2#3}{\int}$}
\vcenter{\hbox{$#2#3$}}\kern-.5\wd0}}

\newcommand{\R}{\mathbb{R}}
\newcommand{\N}{\mathbb{N}}
\newcommand{\C}{\mathcal}
\newcommand{\ol}{\overline}
\newcommand{\dx}{\,\mathrm dx}
\newcommand{\dy}{\,\mathrm dy}
\newcommand{\dt}{\,\mathrm dt}
\newcommand{\ds}{\,\mathrm ds}
\newcommand{\dr}{\,\mathrm dr}
\newcommand{\dxt}{\,\mathrm dx\,\mathrm dt}
\newcommand{\dxs}{\,\mathrm dx\,\mathrm ds}

\newcommand{\weaklim}{\rightharpoonup}

\newcommand{\interior}{\mathrm{int}}

\newcommand{\essinf}{\mathrm{ess\,inf}}
\newcommand{\compact}{\subset\!\subset}
\newcommand{\var}{\mathrm{var}}
\newcommand{\essvar}{\mathrm{ess\,var}}

\allowdisplaybreaks

\begin{document}
\begin{center}
	\Large
	Complete damage in linear elastic materials 
	\\ -- \\Modeling, weak formulation and existence results
\end{center}
\begin{center}
	Christian Heinemann\footnote{Weierstrass Institute for Applied Analysis and Stochastics (WIAS), Mohrenstr. 39, 10117 Berlin,\\
		This project is supported by the DFG Research Center ``Mathematics for Key Technologies''  Matheon in Berlin.},
	Christiane Kraus$^1$
\end{center}
\vspace{2mm}
\noindent

\begin{abstract}
The analysis of material models which allow for complete damage is
of major interest in material sciences and has received an increasing attraction in the recent years.
In this work, we study a degenerating evolution inclusion describing
complete damage processes coupled with a quasi-static force balance equation and mixed boundary conditions.
For a realistic description, the inclusion is considered on a time-dependent domain and degenerates when the material undergoes maximal damage.
We propose a weak formulation where the differential inclusion is translated into a variational inequality in combination with a total energy inequality.
The damage variable is proven to be in a suitable $SBV$-space and the displacement field in a local Sobolev space.
We show that the classical differential inclusion and the boundary conditions can be regained from the notion of weak solutions under
additional regularity assumptions.

The main aim is to prove global-in-time existence of weak solutions 
for the degenerating system by performing a degenerate limit.
The variational inequality in the limit is recaptured by suitable approximation techniques
whereas the energy inequality is gained via $\Gamma$-convergence techniques.
To establish a displacement field for the elastic behavior in the limit,
a rather technical representation result of nonsmooth domains by Lipschitz
domains, which keep track of the Dirichlet boundary, is proven.
\end{abstract}

weak solution, doubly nonlinear
    differential inclusions, existence results, rate-dependent systems.  \\[4mm]

\section{Motivation}
\label{section:motivation}
	From a microscopic point of view, damage behavior originates from breaking atomic links in the material
	whereas a macroscopic theory may specify the damage by a scalar variable related to the quantity of damage.
	According to the latter perspective, phase-field models are quite
	common to model smooth transitions between damaged and undamaged
	material states. Such phase-field models have been mainly
	investigated for incomplete damage.  
	However, for a realistic modeling of damage processes in elastic materials, 
	complete damage theories have to be considered, where the material 
	can completely disintegrate. 

	Mathematical works of complete models covering global-in-time existence are rare and are mainly focused on purely \textit{rate-independent systems}
	\cite{Mielke06, BMR09, Mielke10, Mie11} by using $\Gamma$-convergence techniques to recover energetic properties
	in the limit.
	Existence results for \textit{rate-dependent} complete damage systems in thermoviscoelastic materials are recently published
	in \cite{RR12}.
	In contrast, much mathematical efforts have been made in understanding incomplete damage processes.
	Existence and uniqueness results for damage models of viscoelastic materials are proven in
	\cite{BSS05} in the one dimensional case. Higher dimensional damage
	models and related analytically properties 
	are investigated in \cite{AT90, BS04,  Gia05, MT10, LOS10, KRZ11, BM14} and,
	there, existence, uniqueness,  regularity and approximation results are shown.
	A coupled system describing incomplete damage, linear elasticity and phase separation appeared in \cite{WIAS1520, WIAS1569}.
	All these works are based on the gradient-of-damage model proposed by
        Fr\'emond and Nedjar \cite{FN96} (see also \cite{Fre02}) which
	describes damage as a result from microscopic movements in the solid.
	The distinction between a balance law for the microscopic forces and constitutive relations of the material yields a satisfying derivation
	of an evolution law for the damage propagation from the physical point of view.
	In particular, the gradient of the damage variable enters the resulting equations
	and serves as a regularization term for the mathematical analysis.
	When the evolution of the damage is assumed to be uni-directional, i.e., the damage is irreversible,
	the microforce balance law becomes a differential inclusion. 

	Damage modeling is an active field in the engineering community since the
	1970s. For some recent works, we refer to \cite{Car86, NPO94, Mie95, MK00, MS01, Fre02, LD05, Gee07, VSL11}. A variational 
	approach to fracture and crack propagation models can be found for
	instance in  \cite{ BFM08, CFM09, CFM10, Neg10, LT11}.
	For a non-gradient approach of damage models for brittle materials we refer to
	\cite{FG06,GL09, Bab11}. There, the damage variable $z$ takes on two distinct values,
	i.e. $\{0,1\}$, in contrast to phase-field models where intermediate values $z \in [0,1]$ are also allowed. In addition,
	the mechanical properties of damage phenomena are described in \cite{FG06, GL09, Bab11} differently. They 
	choose a $z$-mixture of a linearly elastic strong and weak material
	with two different elasticity tensors.

	The reason why incomplete damage models are more feasible for mathematical investigations is that
	a coercivity assumption on the free energy prevents the material
	from a complete degeneration and dropping this assumption may lead to serious troubles.
	However, in the case of viscoelastic materials,
	the inertia terms
	circumvent this kind of problem in the sense that
	the deformation field still exists on the whole domain accompanied with a loss of spatial regularity (cf. \cite{RR12}).
	Unfortunately, this result cannot be expected in the case of quasi-static mechanical equilibrium (see for instance \cite{BMR09}).
	
	The main aim of this work is to introduce and analyze a complete damage model
	for linear elastic materials which are assumed to be in quasi-static equilibrium.
	For the analytical discussion, we start with an incomplete damage model
	which is regularized in the equation of balance of forces as in \cite{Mielke06, Mielke10, BMR09, RR12} such that already known
	existence results can be applied. The basis for a weak formulation
	of the regularized system is a notion introduced in \cite{WIAS1520}.
	On the one hand, this weak notion is suitable for incomplete damage PDE systems described on Lipschitz domains and with mixed boundary conditions
	for the deformation field.
	On the other hand, it seems well adapted for the transition to complete damage (see also \cite{RR12}).
	The advantage is that we can deal with low regularity solutions and we
	are able to use weakly semi-continuity arguments for the passage to the limit.
	In our weak framework, the evolution inclusion for the damage process which is classically described by a doubly nonlinear differential inclusion
	becomes a variational inequality combined with a total energy inequality.
	Nevertheless, we are faced with several mathematical challenges since the system highly degenerates during the passage.
		
	The major challenge is to establish a meaningful deformation field on regions where the damage is \textit{not} complete in the limit system.
	For instance, it might happen that in the limit path-connected components of the not completely damaged material are isolated from the
	Dirichlet boundary. In this particular situation, an isolated fragment is surrounded by completely damaged material with degenerated boundary values
	and its evolution is independent of the remaining material
	parts. Therefore, such material parts will be excluded from the
	considered evolution process.
	
	The crucial modeling idea in this paper consists in formulating the PDE system on a \textit{time-dependent domain}.
	The domain contains all the not completely damaged path-connected components of the material which
	are, in a certain sense, connected to the Dirichlet boundary.
	Inside this domain, the damage evolution is driven by a differential inclusion.
	The remaining area of the original domain consists of completely damaged material and of material parts which are not completely damaged and isolated
	from the Dirichlet boundary (see Figure 1).
	Two complicacies arise in this context.
	
	The first issue concerns the energy inequality. The time-dependent domain approach leads to jumps in the
	energy which have to be accounted for in the energy inequality of the notion of weak solutions as well. This issue is tackled with $\Gamma$-convergence techniques
	in order to keep track of the energy at jump points.
	
	Secondly, the time-dependent domain might have very bad smoothness properties which might lead to a failure of Korn's inequality.
	This problem is approached by proving some rather technical covering results for these sets with smooth domains where Korn's inequality can be applied.
	In this context, we introduce some special kind of local Sobolev spaces
	which seem to be the right spaces for looking at solutions in the limit system.
	
	\begin{figure}[b!]
		\begin{center}
			\includegraphics[width=6.4in]{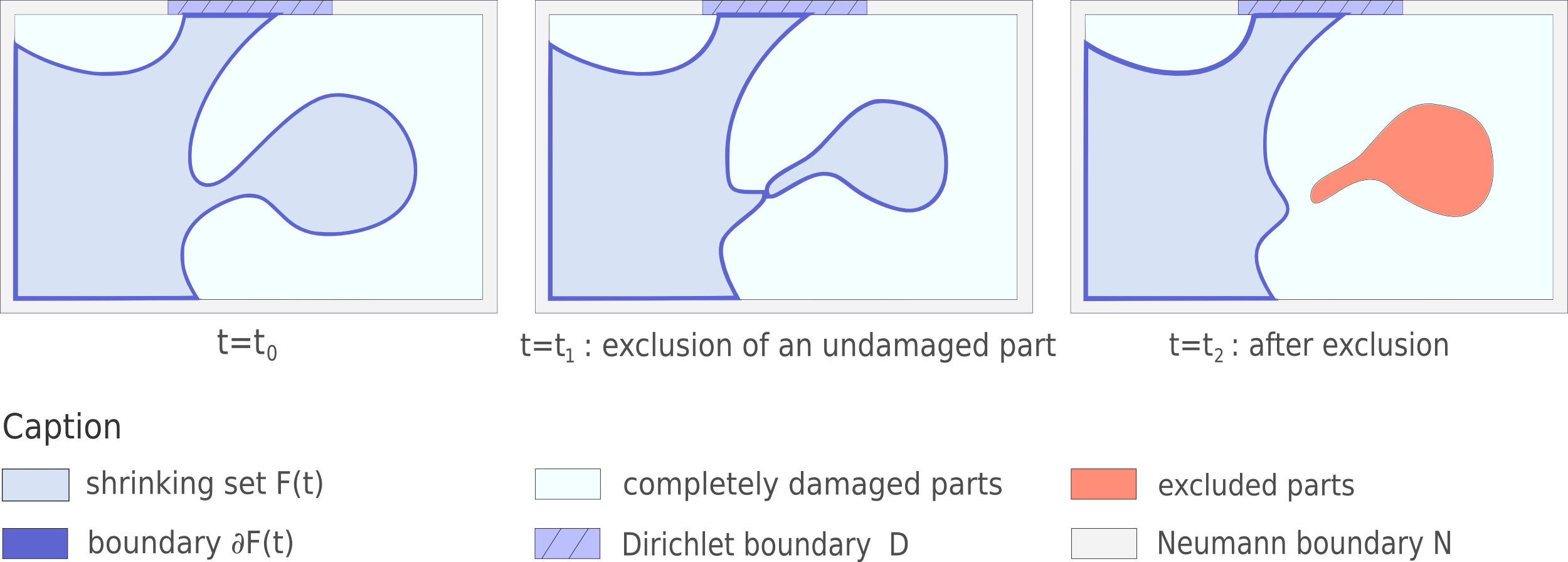} 
		\end{center}
		\caption{\textit{This illustration shows an exclusion process of an undamaged material part in 2D during the evolution.
		The dark blue curve encircles the maximal admissible subset $\mathfrak A_{D}(\{z(t)>0\})$ of the not completely damaged area
		$F(t):=\{z(t)>0\}$.
		}}
		\label{fig:figmodel}
	\end{figure}
	
	This paper is structurized as follows.
	The next section provides an overview of the notation we are going to
        use while Section \ref{section:modeling} develops our model first in a classical setting
	with enough smoothness properties and then in a rigorous mathematical setting by presenting a weak formulation
	with $SBV$-functions for the damage variable and local Sobolev functions for the deformation field.
	It is shown in Theorem \ref{theorem:weakImpliesClassical} that the weak notion reduces to the classical PDE system when enough regularity is assumed.
	The main result, i.e. Theorem \ref{theorem:mainExistenceResult}, is stated in Section \ref{section:mainResult}
	while the proof is carried out in the subsequent Section \ref{section:proof}.
	We first perform a degenerate limit procedure in Section \ref{section:simplProblem}. By Zorn's lemma, global-in-time existence of solutions of weak solutions
	will be proven in Section \ref{section:mainProb}.

	To the best of our knowledge, there are no
	global-in-time existence results in the mathematical literature for complete damage models with quasi-static mechanical forces
	and mixed boundary conditions. In addition, for our proposed model, the classical setting can
        be regained from the weak formulation if the solutions
        are smooth enough, which is novel in the existing theory of complete
        damage models.
	
\section{Notation}
\label{section:notation}
	Let $\Omega\subseteq\R^n$ denote a bounded Lipschitz domain, $D\subseteq\partial\Omega$ be a part of the boundary with $\C H^{n-1}(D)>0$.
	and $[0,T]$ with $T>0$ be the time-interval.
	The following table provides an overview of some elementary notation used in this paper.\vspace*{0.5em}\\
	\begin{tabular}{p{9em}p{32.0em}}
			$\Omega_T$, $D_T$ & \textit{$\Omega\times(0,T)$ and $D\times(0,T)$}\\
			$A:B$&\textit{Euclidean matrix product of $A\in\R^{n\times n}$ and $B\in\R^{n\times n}$}\\
			$(f)^+$ &\textit{non-negative part of $f$, i.e. $\max\{0,f\}$}\\
			$\C L^n$, $\C H^n$ & \textit{$n$-dimensional Lebesgue and Hausdorff measure}\\
			$\mathds 1_A$, $I_A$ & \textit{characteristic function and indicator function $X\rightarrow \R\cup\{\infty\}$ with respect to a subset $A\subseteq X$}\\
			$B_\varepsilon(A)$ & \textit{$\varepsilon$-neighborhood of $A\subseteq\R^n$}\\
			$\ol A,\mathrm{int}(A), \partial A$ & \textit{closure, interior and boundary of $A\subseteq\R^n$}\\
					$\{v=0\}$, $\{v>0\}$ & \textit{level and super-level set of $v$, i.e., $\big\{x\in\ol\Omega\,\big|\,v(x)=0\text{ a.e.}\big\}$
						and $\big\{x\in\ol\Omega\,\big|\,v(x)>0\text{ a.e.}\big\}$ for functions $f\in L^1(\Omega)$
						defined up to a set of measure $0$ and defined uniquely
						if $v\in W^{1,p}(\Omega)$, $p>n$, as $W^{1,p}(\Omega)\hookrightarrow \C C(\ol\Omega)$}\\
					$\C C_x^k(\ol{\Omega_T};\R^N)$ & \textit{space of $k$-times continuously differentiable functions
					with respect to the spatial variable on the set $\Omega\times[0,T]$
					where the $k$-th spatial derivatives can be continuously
					extended to $\ol{\Omega_T}$}\\
			$\partial J$& \textit{subdifferential of a convex function $J:X\rightarrow\R\cup\{\infty\}$, $X$ Banach space}\\
			$\mathrm{supp}(v)$ & \textit{support of a function $v$}\vspace{0.5em}
	\end{tabular}\\
	
	Let $(X,\|\cdot\|)$ be a Banach space, $I\subseteq\R$ be an open interval and $\mu$ be a positive measure.
	The space $L^p(I,\mu;X)$, $1\leq p\leq\infty$, denotes the $p$-Bochner $\mu$-integrable functions with values in $X$
	($\mu$-essentially bounded for $p=\infty$, respectively).
	We write $L^p(I;X)$ for $L^p(I,\C L^1;X)$.
	The subspace $H^{q}(I;X)\subseteq L^2(I;X)$, $q\in\N$, indicates $L^2$-functions which are $q$-times weakly differentiable with weak derivatives in $L^2$.
	Moreover, the subspace $BV(I;X)\subseteq L^1(I;X)$ consists of functions $f\in L^1(I;X)$
	with
	$$
		\essvar_I(f):=\inf\big\{\var_I(g)\,|\,\text{$g=f$ $\C L^1$-a.e. in $I$}\big\}<+\infty,
	$$
	and
	$$
		\var_I(f):=\sup\Big\{\sum_{i=1}^{k-1}\|f(t_{i+1})-f(t_i)\|\,\Big|\,t_1<t_2<\ldots<t_k\text{ with }t_1,t_2,\ldots, t_k\in I\text{ for }k\geq 2\Big\}.
	$$
	To every $f\in BV(I;X)$, we can choose a representant (also denoted by $f$) with $\var_I(f)<+\infty$.
	Then the values $f(t^\pm):=\lim_{s\rightarrow t^\pm}f(s)$ exist for all $t\in \ol{I}$
	(and are independent of the representant)
	by adapting the convention $f\big((\inf I)^-\big):=f\big((\inf I)^+\big)$ and $f\big((\sup I)^+\big):=f\big((\sup I)^-\big)$.
	The functions $f^+(t):=f(t^+)$ and $f^-(t):=f(t^-)$ are thus uniquely defined for every $t\in \ol I$ and do not coincide for at most countably many points,
	i.e., in the jump discontinuity set $J_f$.
	Furthermore, a regular measure $\mathrm df$ with finite variation, i.e. $|\mathrm df|(I)<\infty$, and with values in $X$ (called \textit{differential measure})
	can be assigned such that
	$\mathrm df((a,b])=f^+(b)-f^+(a)$ for all $a,b\in\ol I$ with $a\leq b$, cf. \cite{Din66}.
	If $X$ is a finite dimensional vector space we refer to \cite{Ambrosio00} for a comprehensive introduction.
	
	If $X$ exhibits the Radon-Nikodym property (e.g. if $X$ is reflexive) the differential measure decomposes into $\mathrm df=f'_\mu\mu$
	for a (not necessarily uniquely) positive Radon measure $\mu$ and a function $f'_\mu\in L^1(I,\mu;X)$ \cite{MV87}.
	The subspace $SBV(I;X)\subseteq BV(I;X)$ of special functions of bounded variation is defined as the space of functions
	$f\in BV(I;X)$ where the decomposition
	\begin{align*}
		\mathrm df = f'\C L^1 + (f^+-f^-)\C H^0\lfloor J_f
	\end{align*}
	for an $f'\in L^1(I;X)$ exists.
	This function $f'$ is called the absolutely continuous part of the differential measure
	and we also write $\partial_t^\mathrm{a}f$.
	If, additionally, $\partial_t^\mathrm{a}f\in L^p(I;X)$, $p\geq 1$, we write $f\in SBV^p(I;X)$.
	
	For the analysis of the system given in the next chapter, it is convenient to introduce local Sobolev functions on shrinking sets.
	Let $G\subseteq\ol{\Omega_T}$ be a subset.
	The intersection of $G$ at time $t\in[0,T]$, i.e. $G\cap(\ol\Omega\times\{t\})$, is denoted by $G(t):=\{x\in\ol\Omega\,|\,(x,t)\in G\}$.
	We call $G$ \textit{shrinking} if $G$ is relatively open in $\ol{\Omega_T}$ and $G(s)\subseteq G(t)$ for
	arbitrary $0\leq t\leq s\leq T$.

	In the sequel, $G\subseteq\ol{\Omega_T}$ denotes a shrinking set.
	We define the following time-dependent local Sobolev space:
	\begin{align}
		L_t^2H^q_{x,\mathrm{loc}}(G;\R^N)
			:=\Big\{&v:G\rightarrow\R^N\;\big|\;
			\forall t\in(0,T],\,\forall U\compact G(t)\text{ open}\,:\;
			v|_{U\times (0,t)}\in L^{2}(0,t;H^{q}(U;\R^N))\Big\}.
		\label{eqn:localSobolevSpace}
	\end{align}
	Here, $L_t^2H^0_{x,\mathrm{loc}}(G;\R^N)$ coincides with $L_\mathrm{loc}^{2}(G;\R^N)$ (see Section \ref{subsection:covering}).
	$L_\mathrm{loc}^2(G;\R^N)$ denotes the classical local $L^2$-Lebesgue space on $G$ given by
	$$
		L_\mathrm{loc}^2(G;\R^N):=\Big\{v:G\rightarrow\R^N\;\big|\;\forall V\compact G\text{ open}\,:\;v|_V\in L^{2}(V;\R^N)\Big\}.
	$$
	(Note that we do not demand that $G$ should be open.)
	As usual, we set $L_t^2H^q_{x,\mathrm{loc}}(G):=L_t^2H^q_{x,\mathrm{loc}}(G;\R)$.
	At fixed time points $t\in(0,T)$, we find $v(t)\in H_\mathrm{loc}^{q}(G(t);\R^N)$.
	If $q\geq 1$ we write $\nabla v$ for the weak derivative with respect to the spatial variable
	as well as $\epsilon(v):=\frac12(\nabla v+(\nabla v)^\mathrm{T})$ for its symmetric part.
	The precise definition and characterization of $\nabla v$ can be found in Proposition \ref{prop:localSob}.

	Given $v\in L_t^2H^q_{x,\mathrm{loc}}(G;\R^N)$ with $q\geq 1$,
	we say that $v=b$ on $D_T\cap G$ with $D\subseteq\partial\Omega$ and $\C H^{n-1}(D)>0$ if
	for every $t\in(0,T)$ and every open set $U\compact G(t)$ with Lipschitz boundary
	\begin{align}
	\label{eqn:boundaryValueCondition}
		\widetilde v(s)=b(s)\text{ on }\partial U\cap D\text{ in the sense of traces for a.e. }s\in(0,t),
	\end{align}
	is fulfilled with $\widetilde v:=v|_{U\times (0,t)}\in L^2(0,t;H^1(U;\R^N))$.


\section{Modeling and main results}
\label{section:modeling}
\subsection{Classical formulation}
\label{section:classicalFormulation}

	The damage model we want to study is based on the gradient-of-damage theory (see \cite{FN96}) and makes use of
	the following free energy density $\varphi$ and dissipation potential density $\varrho$:
	\begin{align}
	\label{eqn:freeEnergyFormal}
		\varphi(e,z,\nabla z):=\frac 1p|\nabla z|^p+W(e,z)+h(z),\quad\varrho(\dot z):=-\alpha\dot z+\beta|\dot z|^2+I_{(-\infty,0]}(\dot z),
	\end{align}
	where $e$ denotes the linearized strain tensor,
	$z$ the damage phase-field variable, $\alpha\geq 0$ and $\beta>0$.
	The function $W$ represents the elastic energy density and $h$ is a damage dependent potential.
	The gradient exponent $p$ satisfies $p>n$.
	The damage variable $z$ specifies the degree of damage at each reference position $x\in\Omega$ in the material, i.e.,
	$z(x)=1$ stands for an undamaged and $z(x)=0$ for a completely damaged material point whereas intermediate values represent partial damage.
	Furthermore, the irreversibility of the damage process (the solid can not heal itself) is ensured by the indicator function in $\varrho$.

	The evolution is given by the parabolic-elliptic system
	\begin{subequations}
	\begin{align}
	\label{eqn:microforceBalance}
		&-\mathrm{div}(\varphi_{,\nabla z}(e,z,\nabla z))+\varphi_{,z}(e,z,\nabla z)+\partial I_{[0,1]}(z)+\partial\varrho(\partial_t z)\ni 0,\\
	\label{eqn:forceBalance}
		&-\mathrm{div}(\varphi_{,e}(e,z,\nabla z))=l,
	\end{align}
	\end{subequations}
	where the inclusion \eqref{eqn:microforceBalance} specifies the flow rule for the damage profil
	and the elliptic equation \eqref{eqn:forceBalance} describes a quasi-static mechanical equilibrium of the forces.
	Here, $l$ are the exterior volume forces.
	Note that the two subdifferentials (denoted by $\partial$) appear on right hand side of the inclusion
	to account for the constraints $0\leq z\leq 1$ and $\partial_t z\leq 0$.
	
	This paper will cover elastic energy densities of the form
	\begin{align}
	\label{eqn:definitionW}
		W(e,z)=\frac 12 g(z)\mathbb Ce:e
	\end{align}
	with a symmetric and positive definite stiffness tensor $\mathbb C\in\mathcal L(\mathbb R_\mathrm{sym}^{n\times n})$
	and a function $g\in\C C^1([0,1];\R^+)$ with the properties
	\begin{align}
	\label{eqn:assumptiong}
		\eta\leq g'(z),\qquad g(0)=0
	\end{align}
	for all $z\in [0,1]$ and some constant $\eta>0$.
	
	We use the small strain assumption, i.e., the strain calculates as
	\begin{align}
	\label{eqn:strainDeformationRel}
		e=\epsilon(u):=\frac 12\left(\nabla u+(\nabla u)^\mathrm{T}\right),
	\end{align}
	where the right hand side denotes the symmetric gradient of the displacement field $u$.
	
	
	Note that complete damage occurs if and only if $g(0)=0$. The case $g(0)>0$ would describe incomplete damage which is already covered in the
	mathematical literature (see Section \ref{section:motivation}).
	As mentioned in the introduction, we start with a regularization, where a regularized elastic energy density
	$W^\varepsilon$, $\varepsilon>0$, is used instead of $W$. More precisely, $W^\varepsilon$ is given by
	\begin{align}
	\label{eqn:Wreg}
		W^\varepsilon(e,z)=\frac 12(g(z)+\varepsilon)\mathbb Ce:e.
	\end{align}
	
	In the complete damage regime $\varepsilon=0$, the displacement variable becomes meaningless on material fragment with maximal damage because the
	free energy density vanishes regardless of the values of $u$.
	Therefore, the balance laws \eqref{eqn:microforceBalance} and \eqref{eqn:forceBalance} make obviously only sense pointwise in $\{z>0\}$.
	Beyond that, as already mentioned in the introduction, a phenomenon (in the following called \textit{material exclusion}) might occur:
	
	Suppose that at a specific time point $t$, a path-connected component $P$ (relatively open in $\ol\Omega$) in
	$\{x\in\ol\Omega\,|\,z(x,t)>0\}$ is isolated from the Dirichlet boundary, i.e. $\C H^{n-1}(P\cap D)=0$.
	In this case, path-connected components $P$ of the not completely damaged area $\{z(t)>0\}$ isolated from the
	Dirichlet boundary, i.e. $\C H^{n-1}(P\cap D)=0$, will be excluded in our proposed model
	since the detached parts might be of little interest in the evolution
        process, see Figure 1. 
	This type of modeling leads to an evolution system on a time-dependent
        shrinking domain and motivates the definition of maximal admissible subsets.
	
	\begin{definition}[Admissible subsets of $\ol\Omega$]$\;$
	\label{def:admissibleSubset}
		\begin{enumerate}
			\renewcommand{\labelenumi}{(\roman{enumi})}
			\item
				Let $F\subseteq \ol\Omega$ be a relatively open subset and
				$$
				P_F(x):=\big\{y\in F\,|\,\text{$x$ and $y$ are connected by a path in }F\,\big\}
				$$
				for $x\in F$.
				We say that $F$ is \textit{admissible} with respect to the Dirichlet boundary $D$ if for every $x\in F$
				the condition
				\begin{align*}
					\C H^{n-1}(P_F(x)\cap D)>0
				\end{align*}
				is fulfilled.
				Furthermore, $\frak A_D(F)$ denotes the maximal admissible subset of $F$ with respect to $D$, i.e.,
				\begin{align*}
					\frak A_D(F):=\bigcup \left\{G\subseteq F\,|\,G\text{ is admissible with respect to }D\right\}.
				\end{align*}
			\item
				For a relatively open subset $F\subseteq \ol{\Omega_T}$, the set $\frak A_D(F)$ is given by $(\frak A_D(F))(t):=\frak A_D(F(t))$.
		\end{enumerate}
	\end{definition}
	
	In a nutshell, the evolutionary problem \eqref{eqn:microforceBalance} and \eqref{eqn:forceBalance} is
	considered on a time-dependent domain (a shrinking set) which is, for any time, admissible with respect to $D$.
	The whole evolution problem with its initial-boundary conditions can be summarized within a classical notion in the following way
	(note that, due to the monotonicity of the damage function $z$ with respect to time, the set $F:=\frak A_D\big(\{z>0\}\big)$
	is shrinking).

	\begin{definition}[Classical solution for complete damage]
	\label{def:classicalSolution}
		\hspace*{0.1em}\\
		A pair of functions
		$(u,z)$ with $u\in\mathcal C_x^2(F;\R^n)$, $z\in \C C^2(\ol{\Omega_T};\R)$,
		where the shrinking set $F$ is given by $F:=\mathfrak A_{D}\big(\{z>0\}\big)$
		is called a classical solution to the initial-boundary data $(z^0,b)$ if\\
		\begin{tabular}{ll}
			\begin{minipage}{19.0em}
				{
					\begin{flalign*}
						&\qquad l=-\mathrm{div}(\partial_e W^\mathrm{el}(\epsilon(u),z))&\\
						&\qquad 0\in-\mathrm{div}(|\nabla z|^{p-2}\nabla z)+W_{,z}(\epsilon(u),z)+h'(z)\hspace*{4em}&\\
						&\qquad \qquad -\alpha+\beta\partial_t z+\partial I_{(-\infty,0]}(\partial_t z)&
					\end{flalign*}
					}
			\end{minipage}
			&
			\begin{minipage}{14.8em}
				{
					\begin{subequations}
					\begin{flalign*}
							&\hspace*{-1.5em}\text{ in }F,&\\\notag\\
							&\hspace*{-1.5em}\text{ in }F,&
					\end{flalign*}
					\end{subequations}
				}
			\end{minipage}
		\end{tabular}\\
		and the initial-boundary conditions\\
		\begin{tabular}{ll}
			\begin{minipage}{25.5em}
				{
					\begin{flalign*}
						&\qquad z(0)=z^0&\\
						&\qquad u(t)=b(t)&\\
						&\qquad \partial_e W^\mathrm{el}(\epsilon(u(t)),z(t))\cdot\nu=0&\hspace*{5em}\\
						&\qquad z(t)=0&\\
						&\qquad \nabla z(t)\cdot\nu=0&
					\end{flalign*}
					}
			\end{minipage}
			&
			\begin{minipage}{14.7em}
				{
					\begin{subequations}
					\begin{flalign*}
							&\hspace*{-1.5em}\text{ in }F(0),&\\
							&\hspace*{-1.5em}\text{ on }\Gamma_1(t):=F(t)\cap D,&\\
							&\hspace*{-1.5em}\text{ on }\Gamma_2(t):=F(t)\cap (\partial\Omega\setminus D),&\\
							&\hspace*{-1.5em}\text{ on }\Gamma_3(t):=\partial F(t)\setminus F(t),&\\
							&\hspace*{-1.5em}\text{ on }\Gamma_1(t)\cup\Gamma_2(t),&
					\end{flalign*}
					\end{subequations}
				}
			\end{minipage}
		\end{tabular}\\
	\end{definition}
	
	\begin{remark}
			The time-dependent boundary $\partial F(t)$ disjointly decomposes into $\Gamma_1(t)\cup\Gamma_2(t)\cup\Gamma_3(t)$,
			where $\Gamma_1(t)$ indicates the not completely damaged Dirichlet boundary, $\Gamma_2(t)$ the not completely damaged
			Neumann boundary and $\Gamma_3(t)$ the completely damaged boundary (see Figure 2).
			We have the following types of boundary conditions:
			\begin{align*}
				\Gamma_1(t)\qquad\text{---}\qquad &\text{Dirichlet boundary condition for }u\\
					&\text{Neumann boundary condition for }z\\
				\Gamma_2(t)\qquad\text{---}\qquad &\text{Neumann boundary condition for }u\\
					&\text{Neumann boundary condition for }z\\
				\Gamma_3(t)\qquad\text{---}\qquad &\text{degenerated boundary condition}
			\end{align*}
			On the degenerated boundary, $z$ vanishes (homogeneous Dirichlet boundary condition for $z$) and, therefore,
			if we assume that $\epsilon(u)$ can be continuously extended to $\Gamma_3$ the stress $W_{,e}(\epsilon(u),z)$
			vanishes too.
	\end{remark}
	\begin{figure}[h!]
		\begin{center}
			\includegraphics[width=6.4in]{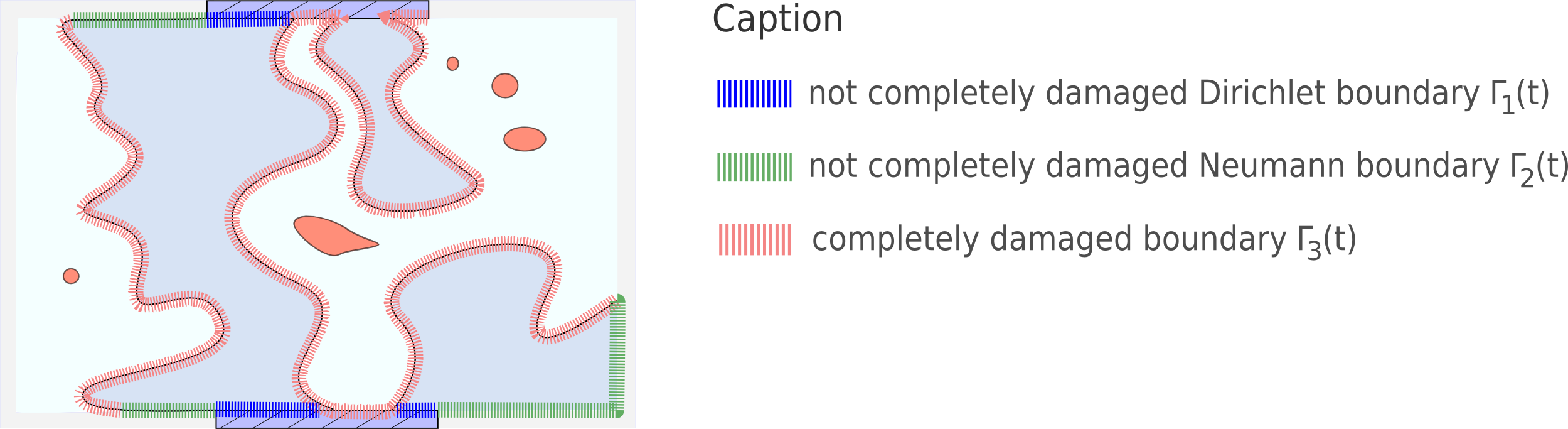} 
		\end{center}
		\caption{\textit{Illustration of the different parts of the boundary of $F(t)$.
		}}
		\label{fig:boundary}
	\end{figure}

	The goal of the next section is to state a suitable weak formulation
        for our introduced complete damage model. 
	Due to the high degree of degeneracy and the non-smoothness of $F$,
	$u$ can only be expected to be in some local Sobolev space on the shrinking set $F$ introduced in \eqref{eqn:localSobolevSpace}.

\subsection{Weak formulation and justification}
	
	For a weak formulation of the system presented in Section \ref{section:classicalFormulation}, we choose the free energy
	$\C E$ whose density has already been given in
        \eqref{eqn:freeEnergyFormal}. To shorten the presentation, we assume
        $h=0$ in  \eqref{eqn:freeEnergyFormal} and  $l=0$ in
        \eqref{eqn:forceBalance}. 
	The analysis in this paper also works for exterior volume forces and $\C C^1$--potential functions $h$ as well.
	
        In contrast to \cite{KRZ11} for incomplete damage models and related works,
	we will not use a purely energetic approach
	but rather a mixed variational/energetic formulation as presented in \cite{WIAS1520}.
	
	\begin{definition}[Free energy]
	\label{definition:freeEnergy}
		Let $e\in L^2(\Omega;\R_\mathrm{sym}^{n\times n})$ and $z\in W^{1,p}(\Omega)$ be given.
		The associated free energy of the system in Definition \ref{def:classicalSolution} is given by
		\begin{align*}
			\C E(e,z):=\int_{\Omega}\left(\frac 1p|\nabla z|^p+W(e,z)\right)\dx,
		\end{align*}
		whereas its $\varepsilon$-regularization with $\varepsilon>0$ (for later use) is defined as (see \eqref{eqn:Wreg})
		\begin{align*}
			\C E_\varepsilon(e,z):=\int_{\Omega}\left(\frac 1p|\nabla z|^p+W^\varepsilon(e,z)\right)\dx.
		\end{align*}
		If $e$ is only defined on a measurable subset $H\subset\Omega$, i.e. $e\in L^2(H;\R_\mathrm{sym}^{n\times n})$,
		we use the convention $\C E(e,z):=\C E(\widetilde e,z)$, where $\widetilde e:=e$ in $H$ and $\widetilde e:=0$ in $\Omega\setminus H$.
	\end{definition}
	We are now able to give a weak formulation of the system
	in an $SBV$-setting (with respect to the damage variable).
	In accordance to Definition \ref{def:classicalSolution}, $z$ is extended on whole $\ol{\Omega_T}$
	and, when viewed as an $SBV^2(0,T;L^2(\Omega))$-function, $z$ has a jump at $t$ if and only if a material exclusion occurs at $t$.
	\begin{definition}[Weak solution] A pair $(u,z)$ is called a weak solution
	of the system given in Definition \ref{def:classicalSolution} with the initial-boundary data $(z^0,b)$ if
	\label{def:weakSolution}
		\begin{enumerate}
			\renewcommand{\labelenumi}{(\roman{enumi})}
			\item
				\textit{Regularity:}
				\begin{align*}
				\begin{aligned}
					&z\in L^\infty(0,T;W^{1,p}(\Omega))\cap SBV^2(0,T;L^2(\Omega)),
					&&u\in L_t^2H^1_{x,\mathrm{loc}}(F;\R^n)
				\end{aligned}
				\end{align*}
				with $\epsilon(u)=:e\in L^2(F;\R_\mathrm{sym}^{n\times n})$
				where $F:=\frak A_D(\{z^->0\})\subseteq\ol{\Omega_T}$ is a shrinking set.
			\item
				\textit{Quasi-static mechanical equilibrium:}
				\begin{align}
				\label{eqn:forceBalanceWeak}
					0=\int_{F(t)}W_{,e}(e(t),z(t)):\epsilon(\zeta)\dx
				\end{align}
				for a.e. $t\in(0,T)$ and 
				for all $\zeta\in H_{D}^1(\Omega;\R^n)$.
				Furthermore, $u=b$ on $D_T\cap F$.
			\item
				\textit{Damage one-sided variational inequality: }
				\begin{align}
				\label{eqn:VI}
					\int_{F(t)}\left(|\nabla z(t)|^{p-2}\nabla z(t)\cdot\nabla\zeta+W_{,z}(e(t),z(t))\zeta\right)\dx
					&\geq\int_{\Omega}(\alpha-\beta\partial_t^\mathrm{a} z(t))\zeta \dx\\
					0&\leq z(t)\text{ in }\Omega,\notag\\
					0&\geq\partial_t^\mathrm{a} z(t)\text{ in }\Omega\notag
				\end{align}
				for a.e. $t\in(0,T)$ and for all $\zeta\in W^{1,p}(\Omega)$ with $\zeta\leq 0$.
				The initial value is given by
				$z^+(0)=z^0$ with $0\leq z^0\leq 1$ in $\Omega$.
			\item
				\textit{Damage jump condition:}
				\begin{align}
				\label{eqn:damageJumpCondition}
					z^+(t)=z^-(t)\mathds 1_{F(t)}\text{ in }\Omega
				\end{align}
				for all $t\in[0,T]$.
			\item
				\textit{Energy inequality: }
				\begin{align}
					&\C E(e(t),z(t))+\int_0^t\int_{F(s)}\left(\alpha|\partial_t^\mathrm{a} z|+\beta|\partial_t^\mathrm{a} z|^2\right)\dxs
					+\sum_{s\in J_{z}\cap(0,t]}\left(\C E^-(s)-\C E^+(s)\right)\notag\\
				\label{eqn:EI}
					&\qquad\leq  \C E^+(0)+\int_0^t\int_{F(s)}W_{,e}(e,z):\epsilon(\partial_t b)\dxs
				\end{align}
				for a.e. $t\in(0,T)$ with 
				$ \C E^-(s):= \lim_{\tau\rightarrow  s^-}\mathop\mathrm{ess\,inf}_{\vartheta\in(\tau,s)}
				\C E(e(\vartheta),z(\vartheta))$ 
				and $\C E^+(s)=\frak e_s^+$, where
				$\frak e_s^+\in\R_+$ is any value satisfying the upper energy
				estimate
				\begin{align}
				\label{eqn:variationalEnergyIneq}
				  \frak e_s^+\leq \C E(\epsilon(b(s)+\zeta),z^+(s))
				\end{align}
				for all $\zeta\in
                                H_\mathrm{loc}^1(\{z^+(s)>0\};\R^n)$ with
                                $\zeta=0$ on $D\cap\{z^+(s)>0\}$.
                                In addition, 
                                $$0\leq \C E^-(s)-\C E^+(s).$$
		\end{enumerate}
	\end{definition}
	\begin{remark}
	\label{remark:weakSolution}
		\begin{enumerate}
			\renewcommand{\labelenumi}{(\roman{enumi})}
                        \item  For all times $t$ we have 
                          $z^-(t) \in W^{1,p}(\Omega)$.  \\ 
                          This property 
                          can be derived as follows. 
                          Let $t\in(0,T]$ be arbitrary. Since $z\in
                            L^\infty(0,T;W^{1,p}(\Omega))$, 
                            we find a sequence $t_n\nearrow t$ such that
		            $\|z(t_n)\|_{W^{1,p}}<C$ for all $n\in\mathbb N$.
		            Because of this and because of $z(t_n)\rightarrow  z^-(t)$ 
                            in $L^2(\Omega)$ as $n\to\infty$ (since $z\in
                            SBV(0,T;L^2(\Omega))$),
                            we even obtain
		            $z(t_n)\rightharpoonup z^-(t)$ in $W^{1,p}(\Omega)$.
		            In particular $z^-(t) \in W^{1,p}(\Omega)$.
                          
			\item
				Lemma \ref{lemma:truncation} and (i) ensure that for all times $t$ we have $z^-(t)\mathds 1_{F(t)}\in W^{1,p}(\Omega)$ (see jump
				condition \eqref{eqn:damageJumpCondition}).
			\item
				Jump condition \eqref{eqn:damageJumpCondition} and the definition of $F$ imply
				$\{z^+(t)>0\}=F(t)$ for all $t\in[0,T]$.
				By the convention introduced in Definition \ref{definition:freeEnergy},
				$$
					\C E(e(t),z(t))=\int_{F(t)}\left(\frac 1p|\nabla z(t)|^p+W(e(t),z(t))\right)\dx,
				$$
				which equals $\int_{\{z(t)>0\}}\left(\frac 1p|\nabla z(t)|^p+W(e(t),z(t))\right)\dx$ for a.e. $t\in(0,T)$.
			\item
				The jump term $\C E^-(s)-\C E^+(s)$ equals the energy of the excluded material parts at time point $s$,
				i.e. $\C E^-(s)-\C E^+(s)=\C E(s^-)-\C E(s^+)$ (for smooth solutions on $F$), where $t\mapsto\C E(e(t),z(t))$ denotes the energy function
				along the trajectory.
				However, for less regular weak solutions as in Definition \ref{def:weakSolution}, the one-sided limits $\C E(s^-)$ and
				$\C E(s^+)$ possibly do not exist as $\C E(t)$
                                is only in $L^\infty(0,T)$.
				But, in any case, $\lim_{\tau\rightarrow s^-}\mathop\mathrm{ess\,inf}_{\vartheta\in (\tau,s)}
				\C E(\vartheta)$ clearly exists and coincides with $\C E(s^-)$ for smooth solutions.
				The value $\C E(s^+)$, on the other hand, can be avoided in a rather indirect way by using upper energy estimates.
				More precisely, it turns out that $\C E(s^+)$ can be substituted by values (denoted by $\frak e_s^+$) merely satisfying
				\eqref{eqn:variationalEnergyIneq}.
				Together with equations \eqref{eqn:forceBalanceWeak}-\eqref{eqn:EI}, $\frak e_s^+$ is forced to coincide
				with $\C E(s^+)$ for smooth solutions.
				This is particularly shown in the proof of the following theorem.
		\end{enumerate}
		
	\end{remark}

	\begin{theorem}
	\label{theorem:weakImpliesClassical}
		Let $(u,z)$ be a weak solution according to Definition \ref{def:weakSolution}.
		We assume the regularity properties
		$u\in\mathcal C^2(\ol{\Omega_T};\R^n)$ with $u=b$ on $D_T$ and
		$z=\widetilde z$ in $F$ for a $\widetilde z \in\mathcal C^2(\ol{\Omega_T};\R)$.
		Then, $(u,\widetilde z)$ is a classical solution according to Definition \ref{def:classicalSolution}.
	\end{theorem}
	\begin{proof}
		We are going to prove the differential inclusion in Definition \ref{def:classicalSolution}.
		The remaining properties follow with much less effort.
		
		The jump condition \eqref{eqn:damageJumpCondition} and the regularity assumption yields for a.e. $(x,t)\in\Omega_T$
		$$
			\partial_t^\mathrm{a} z(x,t)=
			\begin{cases}
				\partial_t z(x,t)&\text{if }(x,t)\in F,\\
				0&\text{if }(x,t)\in {\Omega_T}\setminus F,
			\end{cases}
		$$
		where $\partial_t z(x,t)$ is the classical time-derivative of $z$ at $(x,t)$.
		In the following, we will make use of this property.
		First, observe that by the regularity assumptions
		$q:=(e,z)\in SBV(0,T;X)$ with $X:=L^2(\Omega;\R^{n\times n})\times W^{1,p}(\Omega)$.

		Applying the chain rule (see Corollary \ref{cor:chainRule}) for the continuously Fr\'echet-differentiable energy functional $\C E$ and the
		$X$-valued SBV-function $q$ shows that $\C E\circ q$ is an SBV-function and
		\begin{align*}
			&\C E(q(t^+))- \C E(q(0^+))=\mathrm d (\C E\circ q)\big((0,t]\big)\\
			&\qquad\qquad=\int_0^t\left(\langle\mathrm d_e \C E(q(s)),\partial_t e(s)\rangle
				+\langle\mathrm d_z \C E(q(s)),\partial_t^\mathrm{a} z(s)\rangle\right)\ds\\
			&\qquad\qquad\quad+\sum_{s\in J_{z}\cap(0,t]}\left(\C E(q(s^+))-\C E(q(s^-))\right).
		\end{align*}
		The two terms in the integral on the right hand side can be treated as follows.
		\begin{itemize}
			\item[\textbullet]
				Taking into account $z=0$ in $\Omega_T\setminus F$ and testing \eqref{eqn:forceBalanceWeak} with $\zeta=\partial_t u(s)-\partial_t b(s)$, yields
				\begin{align*}
					\langle\mathrm d_e \C E(q(s)),\partial_t e(s)\rangle&=\int_{\Omega}W_{,e}(\epsilon(u(s)),z(s)):\epsilon(\partial_t u(s))\dx\\
						&=\int_{F(s)}W_{,e}(\epsilon(u(s)),z(s)):\epsilon(\partial_t u(s))\dx\\
						&=\int_{F(s)}W_{,e}(\epsilon(u(s)),z(s)):\epsilon(\partial_t b(s))\dx.
				\end{align*}
			\item[\textbullet]
				Using the property $\partial_t^\mathrm{a} z=0$ in $\Omega_T\setminus F$,
				\begin{align*}
					&\langle\mathrm d_z \C E(q(s)),\partial_t^\mathrm{a} z(s)\rangle\\
					&\qquad\qquad=\int_\Omega\left(|\nabla z(s)|^{p-2}\nabla z(s)\cdot\nabla \partial_t^\mathrm{a} z(s)
						+W_{,z}(\epsilon(u(s)),z(s))\partial_t^\mathrm{a} z(s)\right)\dx\\
					&\qquad\qquad=\int_{F(s)}\left(|\nabla z(s)|^{p-2}\nabla z(s)\cdot\nabla \partial_t{z}(s)+W_{,z}(\epsilon(u(s)),z(s))\partial_t z(s)\right)\dx.
				\end{align*}
		\end{itemize}
		Putting the pieces together, we end up with
		\begin{align}
			&\C E(q(t^+))+\sum_{s\in J_{z}\cap(0,t]}\left(\C E(q(s^-))-\C E(q(s^+))\right)\notag\\
			&\qquad=\C E(q(0^+))+\int_0^t\int_{F(s)}W_{,e}(\epsilon(u),z):\epsilon(\partial_t b)\dxs\notag\\
			&\qquad\quad+\int_0^t\int_{F(s)}\left(|\nabla z|^{p-2}\nabla z\cdot\nabla\partial_t z+W_{,z}(\epsilon(u),z)\partial_t z\right)\dxs.
		\label{eqn:energyIdentity}
		\end{align}
		Note that we have
		$\C E(q(0^+))=\frak e_0^+$.
		Indeed, passing $t\rightarrow0^+$ in \eqref{eqn:EI}
		yields $\C E(q(0^+))\leq\frak e_0^+$.
		The ``$\geq$''-inequality follows from \eqref{eqn:variationalEnergyIneq} tested with $\zeta=u(0)-b(0)$.
		
		Therefore, \eqref{eqn:EI} particularly implies
		\begin{align}
			&\C E(q(t^+))+\int_0^t\int_{F(s)}\left(\alpha|\partial_t z|+\beta|\partial_t z|^2\right)\dxs
			+\sum_{s\in J_{z}\cap(0,t]}\left(\C E^-(s)-\C E^+(s)\right)\notag\\
		\label{eqn:EImod}
			&\qquad\leq \C E(q(0^+))+\int_0^t\int_{F(s)}W_{,e}(e,z):\epsilon(\partial_t b)\dxs
		\end{align}
		Integrating \eqref{eqn:VI} on $[0,t]$ with respect to time, testing it with $\zeta=\partial_t^\mathrm{a} z\leq 0$,
		applying it to \eqref{eqn:energyIdentity} and comparing the result
		with the energy inequality \eqref{eqn:EImod} shows
		\begin{align}
			&\C E(q(t^+))+\sum_{s\in J_{z}\cap(0,t]}\left(\C E(q(s^-))- \C E(q(s^+))\right)
				+\int_0^t\int_{F(s)}\left(-\alpha\partial_t z+\beta|\partial_t z|^2\right)\dxs\notag\\
			&\qquad\qquad\geq \C E(q(0^+))+\int_0^t\int_{F(s)}W_{,e}(\epsilon(u),z):\epsilon(\partial_t b)\dxs\notag\\
			&\qquad\qquad\geq
				 \C E(q(t^+))+\sum_{s\in J_{z}\cap(0,t]}\left(\C E^-(s)-\C E^+(s)\right)
				+\int_0^t\int_{F(s)}\left(-\alpha\partial_t z+\beta|\partial_t z|^2\right)\dxs.
			\label{eqn:bothSideEnergyEst}
		\end{align}
		Taking the definitions of $\C E^-$ and  $\C E^+$ into account and 
                using
		$\C E(q(s^-))=\lim_{\tau\rightarrow s^-}\mathop\mathrm{ess\,inf}_{\vartheta\in (\tau,s)}
		\C E(q(\vartheta))$, estimate \eqref{eqn:bothSideEnergyEst} yields
		\begin{align}
		\label{eqn:energyLessFrakEnergy}
			\sum_{s\in J_{z}} \C E(q(s^+))
				\leq \sum_{s\in J_{z}}\frak e_s^+.
		\end{align}
		On the other hand, by \eqref{eqn:variationalEnergyIneq}, we find
		$\frak e_s^+\leq\C E(q(s^+))$ for all $s\in J_z$.
		Combining this with \eqref{eqn:energyLessFrakEnergy} shows $\C E(q(s^+))=\frak e_s^+$ for all $s\in J_{z}$.
		
		Therefore,
		$\C E^-(s)-\C E^+(s)=\C E(q(s^-))-\C E(q(s^+))$
		and \eqref{eqn:bothSideEnergyEst} becomes an equality.
		Taking also \eqref{eqn:energyIdentity} into account gives
		\begin{align*}
			0&=\int_0^t\int_{F(s)}\left(|\nabla z|^{p-2}\nabla z\cdot\nabla\partial_t z+W_{,z}(\epsilon(u),z)\partial_t z
				-\alpha\partial_t z+\beta|\partial_t z|^2\right)\dxs.
		\end{align*}
		Together with the variational inequality \eqref{eqn:VI} and the regularity assumptions, we obtain
		\begin{align*}
		\begin{split}
			0\leq \int_{F(s)}&\Big(-\mathrm{div}(|\nabla z(s)|^{p-2}
				\nabla z(s))+W_{,z}(\epsilon(u(s)),z(s))-\alpha+\beta\partial_t z(s)\Big)(\zeta-\partial_t z(s))\dx
		\end{split}
		\end{align*}
		for all $s\in(0,T)$ and for all $\zeta\in L^1(F(s))$ with $\zeta\leq 0$.
		This leads to
		\begin{align*}
			0\leq \Big(-\mathrm{div}(|\nabla z|^{p-2}\nabla z)+W_{,z}(\epsilon(u),z)-\alpha+\beta\partial_t z\Big)(\zeta-\partial_t z)
		\end{align*}
		for a.e. $(x,t)\in F$. By the regularity assumptions, this inequality holds pointwise in $F$.
		Therefore, the differential inclusion in Definition \ref{def:classicalSolution} (ii) is shown.
		\ep
	\end{proof}

	\begin{tabular}{ll}
		\hspace*{-2.4em}
		\begin{minipage}{19em}
			{
				One main goal of this work is to prove existence of weak solutions according to Definition \ref{def:weakSolution}.
				Due to the application of Zorn's lemma used in the global existence proof, analytical problems arises when infinitely
				many exclusions of material parts occur in arbitrary short time intervals in the ``future`` (see Figure 3),
				i.e., cluster points from the right of the jump set $J_{z^{\star}}$
				(denoted by $C_{z^{\star}}$ in the following) where 
				$z^{\star}\in SBV(0,T;L^2(\Omega))$ is given by $z^{\star}(t):=z(t)\mathds 1_{\frak A_D(\{z^-(t)>0\})}$.
				In this case, we can show that the shrinking set $F$ can be represented by $\frak A_D(\{z^->0\})$
				up to an arbitrarily small ''error''.
				In addition, the strain $e$ can be identified
				as the symmetric gradient of $u$ in $\frak A_D(F)$.\vspace*{0.5em}
				}
		\end{minipage}
		&
		\begin{minipage}{23em}
			{
				\begin{center}
					\includegraphics[width=3.4in]{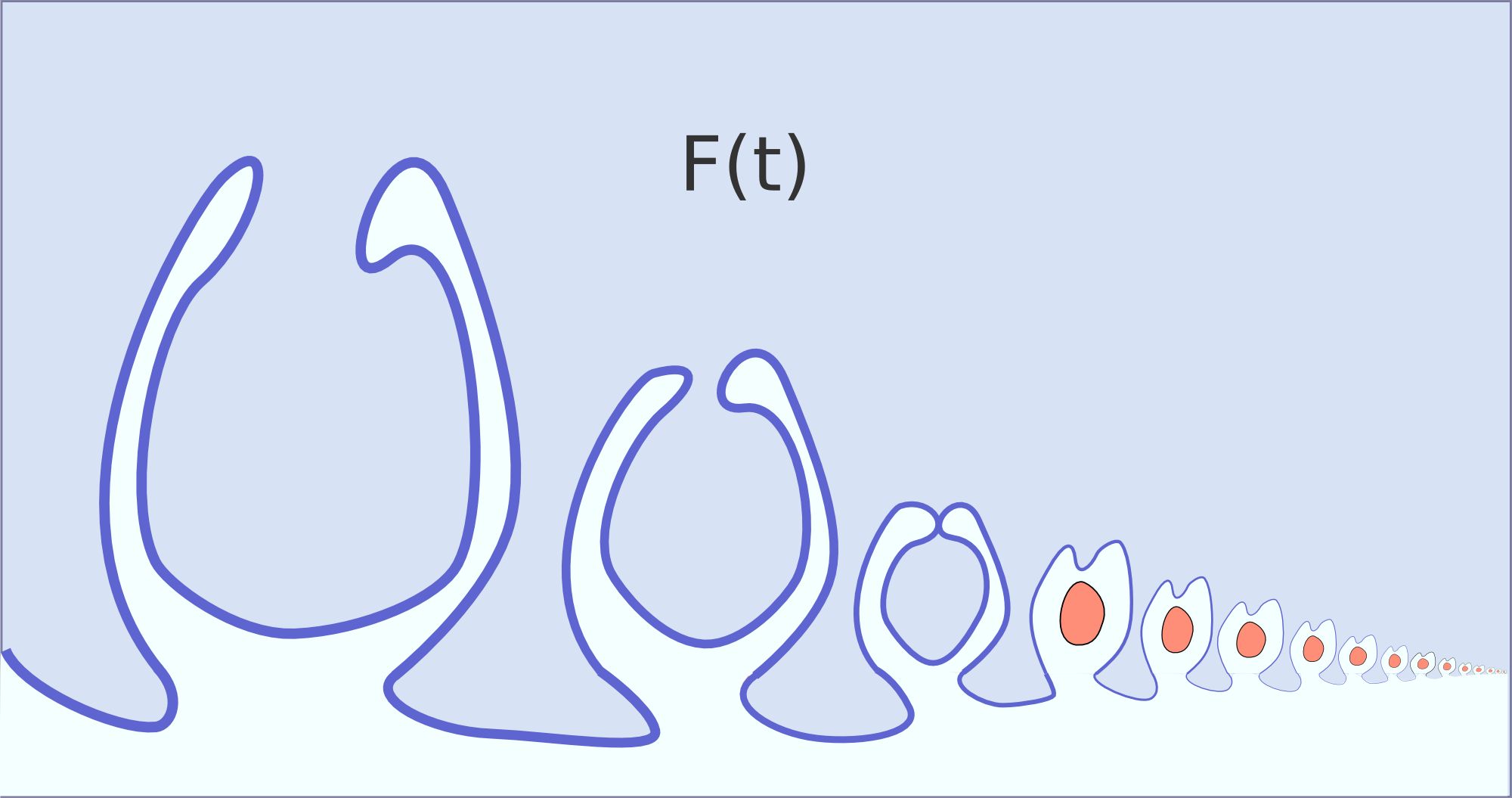} 
				\end{center}
				Figure 3: \textit{An example of a shrinking set where infinitely many exclusions during an ''infinitesimal`` time-interval have occurred.}
			}
		\end{minipage}
	\end{tabular}
	
	To be precise, we introduce the following notion.

	\begin{definition}[Weak solution with fineness ${\bm \eta}$]
		A triple $(e,u,z)$ and a shrinking set $F\subseteq\ol{\Omega_T}$ is called a weak solution with fineness $\eta>0$
		of the system according to Definition \ref{def:classicalSolution} to the initial-boundary data $(z^0,b)$ if
		\label{def:approxWeakSolution}
		\begin{enumerate}
			\renewcommand{\labelenumi}{(\roman{enumi})}
			\item
				\textit{Regularity:}
				\begin{align*}
				\begin{aligned}
					&z\in L^\infty(0,T;W^{1,p}(\Omega))\cap SBV^2(0,T;L^2(\Omega)),
					&&u\in L_t^2H^1_{x,\mathrm{loc}}(\frak A_D(F);\R^n),\\
					&e\in L^2(F;\R_\mathrm{sym}^{n\times n})
				\end{aligned}
				\end{align*}
				with $e=\epsilon(u)$ in $\frak A_D(F)$.
			\item
				\textit{Shrinking set properties:}
				\begin{align*}
				\begin{split}
					&F(t)\supseteq\frak A_D(\{z^-(t)>0\})\text{ for all }t\in[0,T],\\
					&F(t)=\frak A_D(\{z^-(t)>0\})\text{ for all } t\in[0,T]\setminus \bigcup_{t\in C_{z^{\star}}}[t,t+\eta),\\
					&\C L^n\big(F(t)\setminus \frak A_D(\{z^-(t)>0\})\big)<\eta
						\text{ for all } t\in \bigcup_{t\in C_{z^{\star}}}[t,t+\eta).
				\end{split}
				\end{align*}
			\item
				\textit{Evolutionary equations:}
	
				Properties (ii)-(v) of Definition \ref{def:weakSolution} are satisfied.
		\end{enumerate}
	\end{definition}
	\begin{remark}
		If a weak solution $(e,u,z)$ on $F$ with fineness $\eta$ according to Definition \ref{def:approxWeakSolution}
		satisfies $C_{z^{\star}}=\emptyset$ then $(u,z)$ is a weak solution according to Definition \ref{def:weakSolution}.
	\end{remark}

\subsection{Main result}
\label{section:mainResult}
	
	The main result in this work is stated in the following theorem.
	\begin{theorem}[Global-in-time existence of weak solutions with
            fineness ${\bm \eta}$]
	\label{theorem:mainExistenceResult}
		\hspace{0.01em}\\
		Let
		$b\in W^{1,1}(0,T;W^{1,\infty}(\Omega;\R^n))$ and
		$z^0\in W^{1,p}(\Omega)$ with $0\leq z^0\leq 1$ in $\Omega$ and $\{z^0>0\}$, admissible with respect to $D$, be
		initial-boundary data.
		Furthermore, let $\eta>0$ and $W$ be given by \eqref{eqn:definitionW} satisfying \eqref{eqn:assumptiong}.
		Then there exists a weak solution $(e,u,z)$ with fineness $\eta>0$ according
		to Definition \ref{def:approxWeakSolution}.
	\end{theorem}

	\begin{tabular}{ll}
		\hspace*{-2.4em}
		\begin{minipage}{16em}
			{
				The general idea behind the global existence proof is illustrated in Figure 4.
				Starting from an initial damage profile $z^0$ at time $t_0=0$, we calculate a degenerate limit solution via Proposition \ref{prop:existenceResultSimplified}.
				Suppose that the first material exclusions occur at time $t_1$.
				We define a new initial condition $z^1$ where the excluded fragments in $z(t_1)$ are set to $0$ and we calculate a degenerate limit again.
				By repeating this procedure, we obtain
                                $t_0\leq t_1\leq t_2\leq\ldots$\, . 
				The degenerate limit functions on $(t_i,t_{i+1})$ can be concatenated to a weak solution on $(t_0,\sup_{i\in\N}t_i)$ with
				a jump term in the energy inequality at each time $t_i$.
				
				However, to avoid the problem of infinitely many material exclusions occurring in arbitrary short time intervals in the future (see Figure 3),
				''small`` material fragments described by the fineness constant will be neglected in this case.
				The extension to a global-in-time weak solution will be accomplished by Zorn's lemma.
				\vspace*{-0.7em}
				}
		\end{minipage}
		&
		\begin{minipage}{26em}
			{
				\begin{center}
					\includegraphics[width=3.9in]{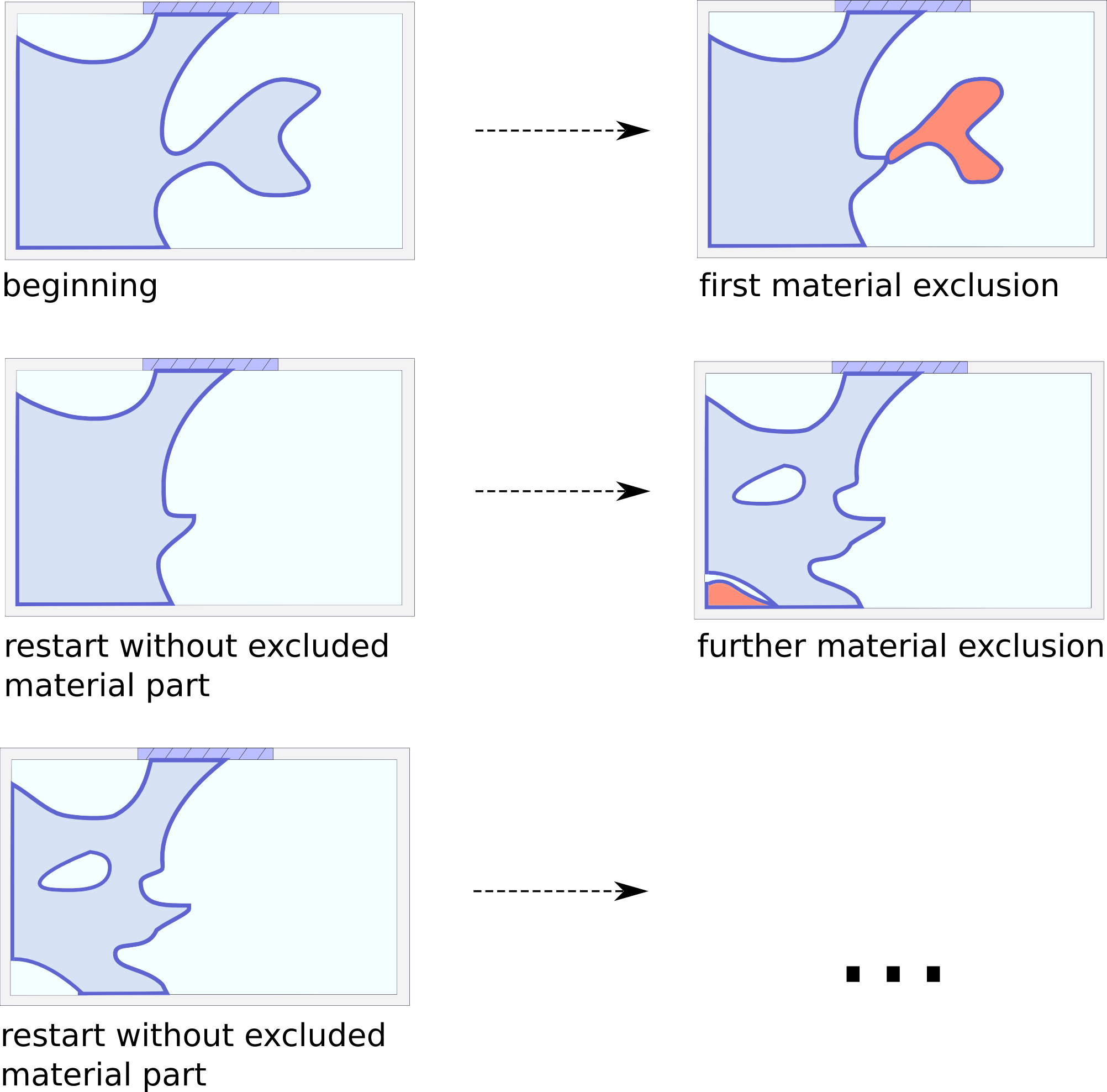}
				\end{center}
				Figure 4: \textit{Concatenation of solutions from the degenerate limit.}
			}
		\end{minipage}
	\end{tabular}

	\begin{remark}
	\label{theorem:localExistence}
		If the initial value for the damage profile contains no complete damaged parts on $\ol\Omega$, we are also able
		to establish maximal local-in-time existence of weak solutions. More precisely,
		there exist a maximal value $\widehat T>0$ with $\widehat T\leq T$ and functions $u$ and $z$ defined on the time interval $[0,\widehat T]$
		such that $(u,z)$ is a weak solution according to Definition \ref{def:weakSolution}.
		Therefore, if $\widehat T<T$, $(u,z)$ cannot be extended
		to a weak solution on $[0,\widehat T+\varepsilon]$.
	\end{remark}

\section{Proof of the main result}
\label{section:proof}
\subsection{Preliminaries}
\subsubsection{Representation and covering properties}
\label{subsection:covering}
	The aim in this subsection is to prove covering results for shrinking sets.
	\begin{definition}[Fine representation]
	\label{def:fineCover}
		Let $H\subseteq\ol\Omega$ be a relatively open subset. We call a countable family $\{U_k\}$
		of open sets $U_k\compact H$ a fine representation for $H$ if
		for every $x\in H$ there exist an open set $U\subseteq\R^n$ with $x\in U$ and a $k\in\N$ such that
		$U\cap\Omega\subseteq U_k$.
	\end{definition}
	
	\begin{remark}
		Note that $H\cap\partial\Omega$ is not covered by $\{U_k\}$. See Figure 5 for an example.
	\end{remark}
	
	\begin{lemma}
	\label{lemma:finiteCover}
		Let $G\subseteq\ol{\Omega_T}$ be a relatively open subset and the sequence $\{t_m\}$ containing $T$ be dense in $[0,T]$.
		Furthermore, let $\{U_k^m\}_{k\in\N}$ be a fine representation for $G(t_m)$ for every
		$m\in\N$.
		Then, for every compact set $K\subseteq G$ there exist a finite set $I\subseteq\N$ and values $m_k\in\N$, $k\in I$, such that
		$K\cap\Omega_T\subseteq\bigcup_{k\in I}\big(U_k^{m_k}\times(0,t_{m_k})\big)$.
	\end{lemma}
	\begin{proof}
		To every element $p=(x,t)\in K$, we will construct a neighborhood $\varTheta_p\subseteq \ol{\Omega_T}$ of $p$
		in the subspace topology of $\ol{\Omega_T}$ such that
		there exists $k,m\in\N$ with $\varTheta_p\cap\Omega_T\subseteq U_k^m\times(0,t_m)$.
		Then the claim follows by the Heine-Borel theorem.
		
		Indeed, to every $p=(x,t)\in K$ there exists an $\varepsilon>0$ such that $B_\varepsilon(p)\cap\ol{\Omega_T}\subseteq G$
		since $G\subseteq\ol{\Omega_T}$ is relatively open.
		Therefore, if $t<T$, $(x,t_m)\in G$ for all $m\in\N$ such that $t<t_m<t+\varepsilon$.
		This implies $(x,t_m)\in G\cap (\ol\Omega\times\{t_m\})=G(t_m)\times\{t_m\}$.
		Then, we find $p\in G(t_m)\times J$ with $J=[0,t_m)$.
		In the case $t=T$, it holds $p\in G(T)\times J$ with $J=[0,T]$.
		Since $\{U_k^m\}_{k\in\N}$ is a fine representation of $G(t_m)$,
		let $\delta>0$ such that $B_{\delta}(x)\cap\Omega\subseteq U_k^m$ for some $k\in\N$.
		Finally, $\varTheta_p:=(B_{\delta}(x)\cap\ol\Omega)\times J$ is the required neighborhood of $p$.
		\ep
	\end{proof}
	
	{
		\begin{center}
			\includegraphics[width=5in]{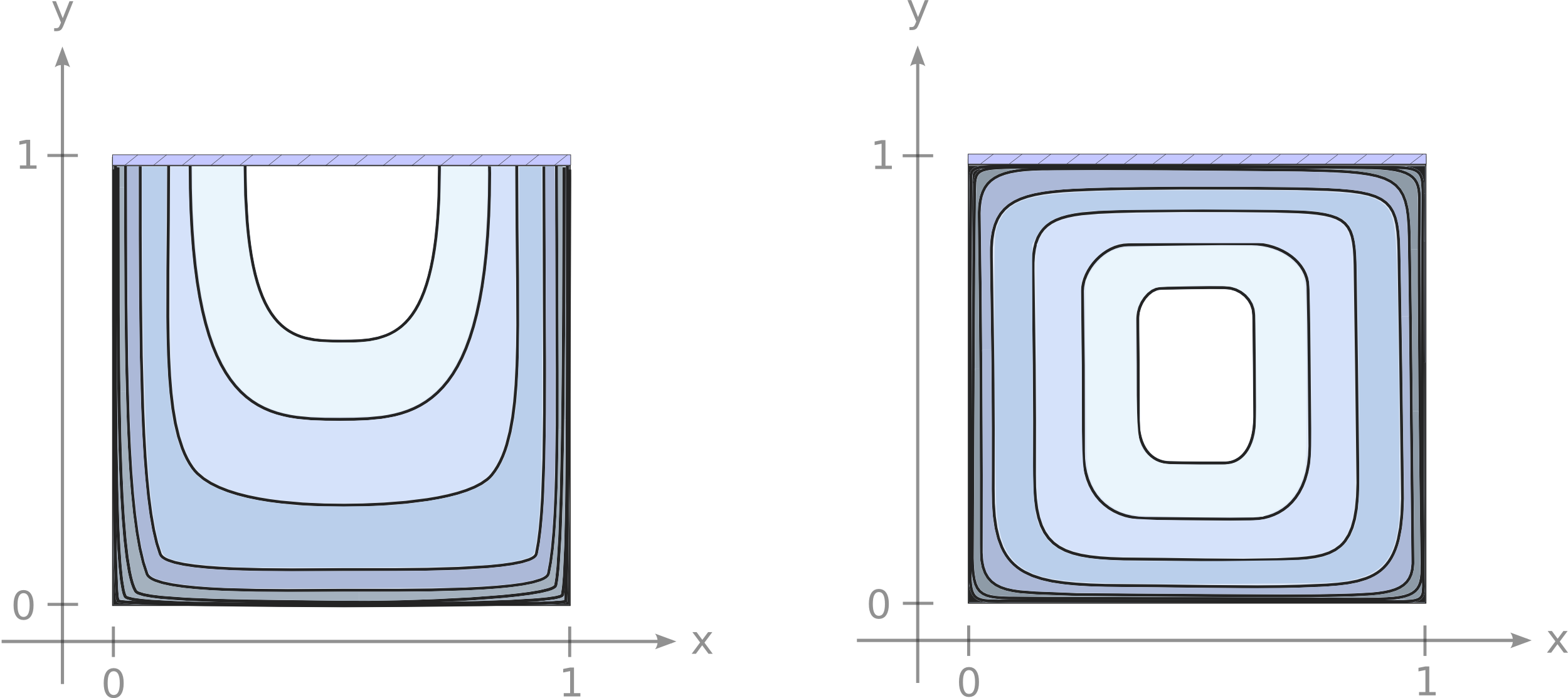}
		\end{center}
		Figure 5: \textit{
		The left illustration shows a fine representation for the relatively open subset $H=(0,1)\times(0,1]$
		of $\ol\Omega=[0,1]\times[0,1]$ whereas the right picture does not show a fine representation for $H$.}
	}

	A simple consequence of Lemma \ref{lemma:finiteCover} is
	$
		L_t^2 H_{x,\mathrm{loc}}^0(G;\R^N)=L_\mathrm{loc}^2(G;\R^N),
	$
	provided that $G\subseteq\ol{\Omega_T}$ is shrinking.
	Moreover, we can characterize the function space $L_t^2H^1_{x,\mathrm{loc}}(G;\R^N)$ as follows.
	\begin{proposition}
	\label{prop:localSob}
		Let $G\subseteq\ol{\Omega_T}$ be a shrinking subset and let $\{t_m\}$ and $\{U_k^m\}$ be as in Lemma \ref{lemma:finiteCover}.
		Furthermore, let $v:G\rightarrow\R^N$ be a function.
		\begin{enumerate}
			\item[(a)]
				The following statements are equivalent:
				\begin{enumerate}
					\item[(i)]
						$v\in L_t^2H^1_{x,\mathrm{loc}}(G;\R^N)$
					\item[(ii)]
						$v|_{U_k^m\times (0,t_{m})}\in L^{2}(0,t_m;H^{1}(U_k^m;\R^N))$ for all $k,m\in\N$
					\item[(iii)]
						$v\in L_\mathrm{loc}^2(G;\R^N)$ and there exists a function $g\in L_\mathrm{loc}^2(G;\R^{N\times n})$ such that
						\begin{align}
						\label{eqn:partIntegration}
							\int_G v\cdot\mathrm{div}(\zeta)\dxt=-\int_G g:\zeta\dxt
						\end{align}
						for all $\zeta\in\C C_\mathrm{c}^\infty(\interior(G);\R^{N\times n})$
				\end{enumerate}
				If one of these conditions is satisfied we write $\nabla v:=g$ and $\epsilon(v):=\frac 12(\nabla v+(\nabla v)^\mathrm{t})$.
			\item[(b)]
				Assume that each $U_k^m$ has a Lipschitz boundary.
				Then the following statements are equivalent:
				\begin{enumerate}
					\item[(i)]
						 $v=b$ on the boundary $D_T\cap G$
					\item[(ii)]
						for every $k,m\in\N$, condition \eqref{eqn:boundaryValueCondition} is satisfied for $U=U_k^m$ and $t=t_m$
				\end{enumerate}
		\end{enumerate}
	\end{proposition}
	\begin{proof}
		\begin{enumerate}
			\item[(a)]
				(i)$\Longrightarrow$(ii) and (iii)$\Longrightarrow$(i) are trivial.
				
				(ii)$\Longrightarrow$(iii):
				Let the function $\widehat g:G\rightarrow \R^{N\times n}$ be $\C L^{n+1}$-a.e. defined as follows.
				For each $k,m\in\N$, we set $\widehat g|_{U_k^m}:=\widehat g_k^m$ where
				$\widehat g_k^m\in L^2(U_k^m\times(0,t_m);\R^{N\times n})$
				is the weak derivative of $v|_{U_k^m\times(0,t_m)}$.
				The function $\widehat g$ is well-defined on $G\cap\Omega_T$ since
				$$
					G\cap \Omega_T=\bigcup_{k,m\in\N}U_k^m\times(0,t_m)
				$$
				and $\widehat g_{k_1}^{m_1}=\widehat g_{k_2}^{m_2}$ on
				$U_{k_1}^{m_1}\times(0,t_{m_1})\cap U_{k_2}^{m_2}\times(0,t_{m_2})$
				for all $k_1,k_2,m_1,m_2\in\N$ in an $\C L^{n+1}$-a.e. sense.
				Let $t\in(0,T]$ and $U\compact G(t)$ be open.
				By Lemma \ref{lemma:finiteCover}, $U\times(0,t)$ can be covered by finitely many sets
				$U_k^{m}\times(0,t_{m})$.
				In particular, $\widehat g|_{U\times(0,t)}\in L^2(0,t;L^2(U;\R^{N\times n}))$.
				Thus $\widehat g\in L_\mathrm{loc}^{2}(G;\R^{N\times n}))$.
				
				Let $\zeta\in\C C_\mathrm{c}^\infty(\interior(G);\R^{N\times n})$.
				Applying Lemma \ref{lemma:finiteCover} again, there exists a finite set $I\subseteq\N$ such that
				$\mathrm{supp}(\zeta)\subseteq \bigcup_{k\in I}U_k^{m_k}\times(0,t_{m_k})=:U$.
				By a partition of unity argument over $U$, \eqref{eqn:partIntegration} holds
				for $g=\widehat g$.
			\item[(b)]
				(ii)$\Longrightarrow$(i):
				Let $t\in(0,T)$ and $U\compact G(t)$ be an arbitrary open subset.
				By Lemma \ref{lemma:finiteCover}, we find a finite set $I\subseteq\N$
				such that $U\subseteq \bigcup_{k\in I}U_k^{m_k}$ and $t_{m_k}\geq t$.
				The claim follows.
				\ep
		\end{enumerate}
	\end{proof}
	If a relatively open set $H\subseteq\ol{\Omega}$ is admissible with respect to $D$
	we can construct a fine representation for $H$ with Lipschitz domains in the following sense.
	\begin{lemma}[Lipschitz representation of admissible sets]
	\label{lemma:LipschitzDomain}
		Let $H\subseteq\ol\Omega$ be relatively open and admissible with respect to $D$.
		Then, there exists a fine representation $\{U_m\}$ for $H$ such that
		\begin{enumerate}
			\item[(i)]
				$U_{m}$ is a Lipschitz domain for all $m\in N$,
			\item[(ii)]
				$\C H^{n-1}(\partial U_{m}\cap D)>0$ for all $m\in N$.
		\end{enumerate}
	\end{lemma}
	\begin{proof}
		We will sketch a possible construction for reader's convenience.
		
		We assume WLOG that $H$ is path-connected because $H$ can only have at most countably many
		path-connected components and for each component we can apply the construction below.
		
		Let us choose a reference point $x_0\in D\cap H$
		with the property
		\begin{align}
		\label{eqn:boundaryGammaCondition}
			\C H^{n-1}\Big(\partial\big(B_\varepsilon(x_0)\cap\Omega\big)\cap D\Big)>0\text{ for all }\varepsilon>0,
		\end{align}
		which is possible since $\C H^{n-1}(D\cap H)>0$.
		The relatively open subset $D_m\subseteq\ol\Omega$ for $m\in\N$ is defined as
		$$
			D_m:=H\setminus \ol{B_{1/m}(\ol\Omega\setminus H)}.
		$$
		If $m$ is large enough we have $x_0\in D_m$ since $H\subseteq\ol\Omega$ is relatively open. We define
		\begin{align*}
			D'_m:=\{x\in D_m\,|\,&x\text{ is path-connected to }x_0\text{ in }D_m\}.
		\end{align*}
		Hence, we obtain an $\varepsilon>0$ such that $B_\varepsilon(x_0)\cap\ol\Omega\subseteq D'_m$ since
		$D'_m$ is relatively open in $\ol\Omega$.
		In combination with \eqref{eqn:boundaryGammaCondition}, this yields $\C H^{n-1}(\partial D'_m\cap D)>0$.
		Because of $D'_m\compact H$, there exists a Lipschitz domain $U_m\subseteq\Omega$ with $D'_m\subseteq \ol{U_m}\subseteq H$
		(e.g. the part of the boundary $\partial U_m\setminus \partial\Omega$ of $U_m$ can be constructed by polygons
		such that $\partial U_m$ fulfills the Lipschitz boundary condition).
		The family $\{U_m\}$ satisfies all the desired properties.
		\ep
	\end{proof}
	\begin{corollary}
	\label{cor:representationFormula}
		Let $G\subseteq\ol{\Omega_T}$ be a shrinking set where $G(t)$ is admissible with respect to $D$ for all $t\in[0,T]$.
		Furthermore, let $\{t_m\}\subseteq[0,T]$ be a dense sequence containing $T$.
		 
		Then, there exists a countable family $\{U_k^{m}\}_{k\in\N}$
		of Lipschitz domains $U_k^m\compact G(t_m)$ for each $m\in\N$ such that
		\begin{enumerate}
			\item[(i)]
				$\C H^{n-1}(\partial U_k^m\cap D)>0$ for all $m\in\N$,
			\item[(ii)]
				$\{U_k^m\}_{k\in\N}$ is a fine representation for $G(t_m)$ for all $m\in\N$,
			\item[(iii)]
				$G=\bigcup_{m=1}^\infty G(t_m)\times [0,t_m]$.
		\end{enumerate}
	\end{corollary}

\subsubsection{$\Gamma$-limit of the regularized energy}
	The construction of the values $\frak e_s^+$ in \eqref{eqn:EI} satisfying the lower energy bound
	\eqref{eqn:variationalEnergyIneq} is based on $\Gamma$-convergence techniques which will be introduced below.
	We refer to \cite{MRS08, BMR09, Ser11} for the utilization of $\Gamma$-convergence
        in the context with rate-independent models and gradient flows.
	
	\begin{definition}[$\Gamma$-limit of the $\varepsilon$-regularized reduced energy]
	\label{def:GammaLimit}
	\hspace{1em}\\
		Let $\frak E_\varepsilon:H^1(\Omega;\R^n)\times W_\mathrm{w}^{1,p}(\Omega)\rightarrow\R_\infty$
		be for $\varepsilon\geq0$ the (regularized) reduced free energy defined by
		\begin{align*}
			\frak E_\varepsilon(\xi,z):=
			\begin{cases}
				\inf_{\zeta\in H_D^1(\Omega;\R^n)}\C E_\varepsilon(\epsilon(\xi+\zeta),z)&\text{if }0\leq z\leq 1,\\
				\infty&\text{else}.
			\end{cases}
		\end{align*}
		Then, we denote by $\frak E$ the $\Gamma$-limit of $\frak E_\varepsilon$ as $\varepsilon\rightarrow 0^+$
		with respect to the topology in $H^1(\Omega;\R^n)\times W_\mathrm{w}^{1,p}(\Omega)$.
		Here, $W_\mathrm{w}^{1,p}(\Omega)$ denotes the space $W^{1,p}(\Omega)$ with its weak topology.
	\end{definition}
	\begin{remark}
		The existence of the $\Gamma$-limit above is ensured because $\{\frak E_\varepsilon\}$ is non-negative and monotonically decreasing as
		$\varepsilon\rightarrow 0^+$.
		Furthermore, $\frak E$ is the lower semi-continuous envelope of $\frak E_0$ in the $H^1(\Omega;\R^n)\times W_\mathrm{w}^{1,p}(\Omega)$ topology
		(see \cite{Braides02}).
	\end{remark}
	To prove properties of the $\Gamma$-limit $\frak E$ which are needed in Section \ref{section:mainProb}, we will establish explicit
	recovery sequences.
	The proof relies on a substitution which is introduced in the following.
	
	Assume that $u\in H^1(\Omega;\R^n)$ minimizes $\C F_\varepsilon(\epsilon(\cdot),z)$ with Dirichlet data $\xi$ on $D$.
	Then, by expressing the elastic energy density $W^\varepsilon$ in terms of its derivative $W_{,e}^\varepsilon$, i.e.
	$W^\varepsilon=\frac 12 W_{,e}^\varepsilon:e$, and
	by testing the Euler-Lagrange equation with $\zeta=u-\widetilde u$
	for a function $\widetilde u\in H^1(\Omega;\R^n)$ with $\widetilde u=\xi$ on $D$, the elastic energy term in
	$\C E_\varepsilon$ can be rewritten as
	\begin{align}
	\label{eqn:weakEnergyTransform}
		&\int_\Omega W^\varepsilon(\epsilon(u),z)\dx
		=\int_\Omega \frac 12(g(z)+\varepsilon)\mathbb C\epsilon(u):\epsilon(\widetilde u)\dx.
	\end{align}

	\begin{lemma}
	\label{lemma:recoverySequence}
		For every $\xi\in H^1(\Omega)$ and $z\in W^{1,p}(\Omega)$
		there exists a sequence $\delta_\varepsilon\rightarrow 0^+$ such that
		$(\xi,(z-\delta_\varepsilon)^+)\rightarrow (\xi,z)$ is a recovery sequence for $\frak F_\varepsilon\xrightarrow{\Gamma}\frak F$
		as $\varepsilon\rightarrow0^+$ where $\frak F$ is the $\Gamma$-limit of
		$\frak F_\varepsilon:H^1(\Omega;\R^n)\times W_\mathrm{w}^{1,p}(\Omega)\rightarrow\R_\infty$
		given by
		$$
			\frak F_\varepsilon(\xi,z):=
			\begin{cases}
				\displaystyle\min_{\zeta\in H_D^1(\Omega;\R^n)}\C F_\varepsilon(\epsilon(\xi+\zeta),z)
				&\text{if }0\leq z\leq 1,\\
				\infty&\text{else.}
			\end{cases}
		$$
		with
		$$
			\C F_\varepsilon(e,z):=\int_{\Omega} W^\varepsilon(e,z)\dx.
		$$
		in the $H^1(\Omega;\R^n)\times W_\mathrm{w}^{1,p}(\Omega)$ topology.
	\end{lemma}
	\begin{proof}
		The $\Gamma$-limit $\frak F$ exists by the same argument as in Definition \ref{def:GammaLimit}.
		Let $(\xi_\varepsilon,z_\varepsilon)\rightarrow (\xi,z)$ be a recovery sequence.
		Since $z_\varepsilon\rightarrow z$ in $\C C^{0,\alpha}(\ol{\Omega})$ due to the compact embedding $W^{1,p}(\Omega)\hookrightarrow\C C^{0,\alpha}(\ol\Omega)$
		for some $0<\alpha<1-\frac np$,
		we can choose a sequence $\delta_\varepsilon\rightarrow 0^+$ such that $(z-\delta_\varepsilon)^+\leq z_\varepsilon$.
		Note that $(z-\delta_\varepsilon)^+\in W^{1,p}(\Omega)$.
		Consider the arrangement
		\begin{align*}
			\frak F_\varepsilon(\xi,(z-\delta_\varepsilon)^+)-\frak F_\varepsilon(\xi_\varepsilon,z_\varepsilon)
			=\underbrace{\frak F_\varepsilon(\xi,(z-\delta_\varepsilon)^+)-\frak F_\varepsilon(\xi,z_\varepsilon)}_{A_\varepsilon}
			+\underbrace{\frak F_\varepsilon(\xi,z_\varepsilon)-\frak F_\varepsilon(\xi_\varepsilon,z_\varepsilon)}_{B_\varepsilon}.
		\end{align*}
		We observe that $A_\varepsilon\leq 0$ because of (note that $(z-\delta_\varepsilon)^+\leq z_\varepsilon$)
		$$
			\C F_\varepsilon(\epsilon(\xi+\zeta),(z-\delta_\varepsilon)^+)
				\leq \C F_\varepsilon(\epsilon(\xi+\zeta),z_\varepsilon)
		$$
		for all $\zeta\in H_D^1(\Omega;\R^n)$.
		Let $u_\varepsilon,v_\varepsilon\in H_D^1(\Omega;\R^n)$ be given by
		\begin{align*}
			&u_\varepsilon=\mathop\mathrm{arg\,min}_{\zeta\in H_D^1(\Omega;\R^n)}\C F_\varepsilon(\epsilon(\xi+\zeta),z_\varepsilon),\\
			&v_\varepsilon=\mathop\mathrm{arg\,min}_{\zeta\in H_D^1(\Omega;\R^n)}\C F_\varepsilon(\epsilon(\xi_\varepsilon+\zeta),z_\varepsilon).
		\end{align*}
		Applying the substitution equation \eqref{eqn:weakEnergyTransform} for $u_\varepsilon$ with $\widetilde u=v_\varepsilon$
		and for $v_\varepsilon$ with $\widetilde u=u_\varepsilon$, we obtain a calculation as follows:
		\begin{align*}
			B_\varepsilon={}&\C F_\varepsilon(\epsilon(\xi+u_\varepsilon),z_\varepsilon)
				-\C F_\varepsilon(\epsilon(\xi_\varepsilon+v_\varepsilon),z_\varepsilon)\\
			={}&\int_{\Omega}\left(\frac 12(g(z_\varepsilon)+\varepsilon)\mathbb C\epsilon(\xi+u_\varepsilon):\epsilon(\xi+v_\varepsilon)
				-\frac 12(g(z_\varepsilon)+\varepsilon)\mathbb C\epsilon(\xi_\varepsilon+v_\varepsilon):\epsilon(\xi_\varepsilon+u_\varepsilon)\right)\dx\\
			\leq{}&\int_\Omega\frac 12(g(z_\varepsilon)+\varepsilon)\Big(\mathbb C\epsilon(\xi):\epsilon(\xi)
				-\mathbb C\epsilon(\xi_\varepsilon):\epsilon(\xi_\varepsilon)\Big)\dx\\
					&+\|\frac 12(g(z_\varepsilon)+\varepsilon)\mathbb C\epsilon(u_\varepsilon+v_\varepsilon)\|_{L^2(\Omega)}\|\epsilon(\xi-\xi_\varepsilon)\|_{L^2(\Omega)}
		\end{align*}
		Using $\xi_\varepsilon\rightarrow \xi$ in $H^1(\Omega)$, $z_\varepsilon\weaklim z$ in $W^{1,p}(\Omega)$
		and the boundedness of
		$\C F_\varepsilon(\epsilon(\xi+u_\varepsilon),z_\varepsilon)$ and 
		$\C F_\varepsilon(\epsilon(\xi_\varepsilon+v_\varepsilon),z_\varepsilon)$ with respect to $\varepsilon$,
		we end up with $\limsup_{\varepsilon\rightarrow0^+}B_\varepsilon\leq 0$.
		Consequently, taking also into account that $(\xi_\varepsilon,z_\varepsilon)\rightarrow (\xi,z)$ is a recovery sequence, we obtain
		$$
			\limsup_{\varepsilon\rightarrow0^+}\frak F_\varepsilon(\xi,(z-\delta_\varepsilon)^+)
				\leq \limsup_{\varepsilon\rightarrow0^+} \frak F_\varepsilon(\xi_\varepsilon,z_\varepsilon)
				+\limsup_{\varepsilon\rightarrow0^+} A_\varepsilon
				+\limsup_{\varepsilon\rightarrow0^+} B_\varepsilon
				\leq \frak F(\xi,z).
		$$
		\ep
	\end{proof}

	\begin{corollary}
	\label{cor:gammaLimitRepresentation}
		\begin{enumerate}
			\renewcommand{\labelenumi}{(\roman{enumi})}
			\item
				For every $\xi\in H^1(\Omega;\R^n)$ and $z\in W^{1,p}(\Omega)$
				\begin{align*}
					\frak E(\xi,z)=\int_\Omega\frac 1p |\nabla z|^p\dx+\frak F(\xi,z).
				\end{align*}
			\item
				The recovery sequence $(\xi,(z-\delta_\varepsilon)^+)\rightarrow(\xi,z)$ for $\frak F_\varepsilon\xrightarrow{\Gamma}\frak F$ constructed in
				Lemma \ref{lemma:recoverySequence} is a recovery sequence for $\frak E_\varepsilon\xrightarrow{\Gamma}\frak E$ as well.
			\item
				Let $\xi\in H^1(\Omega;\R^n)$, $z\in W^{1,p}(\Omega)$
				and $F\subseteq\Omega$ be open such that $\mathds 1_F z\in W^{1,p}(\Omega)$.
				Then $\frak E(\xi,\mathds 1_F z)\leq\frak E(\xi,z)$.
		\end{enumerate}
	\end{corollary}
	\begin{proof}
		\begin{enumerate}
			\renewcommand{\labelenumi}{(\roman{enumi})}
			\item
				Let $(\xi_\varepsilon,z_\varepsilon)\rightarrow (\xi,z)$ be a recovery sequence for $\frak E_\varepsilon\xrightarrow{\Gamma}\frak E$.
				Hence, $\xi_\varepsilon\rightarrow \xi$ in $H^1(\Omega;\R^n)$ and $z_\varepsilon\weaklim z$ in $W^{1,p}(\Omega)$.
				Applying ''$\liminf_{\varepsilon\rightarrow0^+}$`` on each side of the identity
				\begin{align}
				\label{eqn:energyStructure}
					\frak E_\varepsilon(\xi_\varepsilon,z_\varepsilon)
						=\int_\Omega \frac 1p |\nabla z_\varepsilon|^p\dx
						+\frak F_\varepsilon(\xi_\varepsilon,z_\varepsilon)
				\end{align}
				yields for a subsequence
				$$
					\frak E(\xi,z)
						\geq\int_\Omega \frac 1p |\nabla z|^p\dx+\frak F(\xi,z).
				$$
				
				The ''$\leq$`` - part can be shown by considering a recovery sequence $(\xi,(z-\varepsilon)^+)\rightarrow (\xi,z)$ for
				$\frak F_\varepsilon\xrightarrow{\Gamma}\frak F$ according to Lemma \ref{lemma:recoverySequence}
				and applying ''$\liminf_{\varepsilon\rightarrow 0^+}$`` in \eqref{eqn:energyStructure} with
				$(\xi_\varepsilon,z_\varepsilon)=(\xi,(z-\varepsilon)^+)$ on both sides.
			\item
				This follows from (i).
			\item
				Without loss of generality, we assume $0\leq z\leq 1$ on $\Omega$.
				Let $(\xi,(z-\delta_\varepsilon)^+)\rightarrow(\xi,z)$ be a recovery sequence for
				$\frak E_\varepsilon\xrightarrow{\Gamma}\frak E$ as in (ii).
				By assumption, $\mathds 1_F(z-\delta_\varepsilon)^+\in W^{1,p}(\Omega)$ and
				$\mathds 1_F(z-\delta_\varepsilon)^+\rightarrow \mathds 1_F z$ in $W^{1,p}(\Omega)$ as $\varepsilon\rightarrow 0^+$.
				
				Since $\C E_\varepsilon(\epsilon(\xi+\zeta),\mathds 1_F (z-\delta_\varepsilon)^+)
				\leq \C E_\varepsilon(\epsilon(\xi+\zeta),(z-\delta_\varepsilon)^+)$
				for all $\zeta\in H_D^1(\Omega;\R^n)$, we obtain
				\begin{align*}
					\inf_{\zeta\in H_D^1(\Omega;\R^n)}\C E_\varepsilon(\epsilon(\xi+\zeta),\mathds 1_F (z-\delta_\varepsilon)^+)
					&\leq \inf_{\zeta\in H_D^1(\Omega;\R^n)}\C E_\varepsilon(\epsilon(\xi+\zeta),(z-\delta_\varepsilon)^+).
				\end{align*}
				Therefore,
				\begin{align*}
					\frak E_\varepsilon(\xi,\mathds 1_F (z-\delta_\varepsilon)^+)&\leq\frak E_\varepsilon(\xi,(z-\delta_\varepsilon)^+).
				\end{align*}
				Passing to $\varepsilon\rightarrow 0^+$ yields the claim.
			 \ep
		\end{enumerate}
	\end{proof}
	
	\begin{lemma}
	\label{lemma:EfrakEestimate}
		Let $\xi\in H^1(\Omega;\R^n)$ and $z\in W^{1,p}(\Omega)$
		with $0\leq z\leq 1$.
		Furthermore, let $u\in H_\mathrm{loc}^1(\{z>0\};\R^n)$ and 
		for every Lipschitz domain $U\compact \{z>0\}$, $u=\xi$ on $D\cap\partial U$ in the sense of traces.
		Then
		\begin{align*}
			\frak E(\xi,z)\leq \C E(\epsilon(u),z).
		\end{align*}
	\end{lemma}
	\begin{proof}
		Consider an arbitrary $\varepsilon>0$ and define $z_\varepsilon:=(z-\varepsilon)^+$.
		Since $z\in \C C(\ol\Omega)$,
		it holds the compact inclusion $\{z_\varepsilon>0\}\compact \{z>0\}$.
		There exists an open set $U$ with Lipschitz boundary such that
		$\{z_\varepsilon>0\}\subseteq \overline  U\subseteq\{z>0\}$ (e.g. construction of $\partial U\setminus \partial\Omega$ by
		polygons such that $\partial U$ fulfills the Lipschitz boundary condition).
		
		Now, we have $u|_U\in H^1(U;\R^n)$ as well as $u=\xi$ on $\partial U\cap D$.
		There exists an extension $u_\varepsilon\in H^1(\Omega;\R^n)$ with $u_\varepsilon|_U=u|_U$
		and $u_\varepsilon=\xi$ on $D$.
		The monotonicity of $\{\frak E_\varepsilon\}$ with respect to $\varepsilon$ implies that $\frak E$
		is the lower semi-continuous envelope of
		$\widetilde{\frak E}(\xi,z):=\inf_{\varepsilon>0}\frak E_\varepsilon(\xi,z)$.
		in the $H^1(\Omega;\R^n)\times W_\mathrm{w}^{1,p}(\Omega)$-topology
		(cf. \cite{Braides02}).
		By switching the infima, it holds
		$$
			\widetilde{\frak E}(\xi,z)=
			\begin{cases}
				\inf_{\zeta\in H_D^1(\Omega;\R^n)}\C E(\epsilon(\xi+\zeta),z)&\text{if }0\leq z\leq 1,\\
				\infty&\text{else.}
			\end{cases}
		$$
		Since $u=u_\varepsilon$ on $\{z_\varepsilon>0\}$, we get
		\begin{align*}
			\frak E(\xi,z)={}&
				\inf_{\xi_\varepsilon\rightarrow \xi\text{ in }H^1(\Omega;\R^n)}\;
				\inf_{\eta_\varepsilon\weaklim z\text{ in }W^{1,p}(\Omega)}\;
				\liminf_{\varepsilon\rightarrow 0}\;
				\widetilde{\frak E}(\xi_\varepsilon,\eta_\varepsilon)\\
			\leq{}&
				\liminf_{\varepsilon\rightarrow 0}\;\widetilde{\frak E}(\xi,z_\varepsilon)
			\leq
				\liminf_{\varepsilon\rightarrow 0}\;\C E(\epsilon(u_\varepsilon),z_\varepsilon)\\
			\leq{}&
				\liminf_{\varepsilon\rightarrow 0}\;\C E(\epsilon(u),z_\varepsilon)
			=\C E(\epsilon(u),z).
		\end{align*}
		\ep
	\end{proof}
	
\subsection{Degenerate limit}
	\label{section:simplProblem}
	
	In the first step of the proof of Theorem \ref{theorem:mainExistenceResult}, an existence result of a simplified problem,
	where no exclusion of material parts are considered, will be shown.
	The statement we are going to prove in this subsection is given as follows.
	\begin{proposition}[Degenerate limit]
	\label{prop:existenceResultSimplified}
		Let $b\in W^{1,1}(0,T;W^{1,\infty}(\Omega;\R^n))$ and
		$z^0\in W^{1,p}(\Omega)$ with $0\leq z^0\leq 1$ and $\{z^0>0\}$ admissible with respect to $D$
		be initial-boundary data
		and let $W$ be given by \eqref{eqn:definitionW} satisfying \eqref{eqn:assumptiong}.
		Then there exist functions
		\begin{align*}
		\begin{aligned}
			&z\in L^\infty(0,T;W^{1,p}(\Omega))\cap H^1(0,T;L^2(\Omega)),
			&&u\in L_t^2H^1_{x,\mathrm{loc}}(\frak A_D(\{z>0\});\R^n),\\
			&e\in L^2(\{z>0\};\R_\mathrm{sym}^{n\times n})
		\end{aligned}
		\end{align*}
		with $e=\epsilon(u)$ in $\frak A_D(\{z>0\})$
		such that the properties (ii)-(v) of Definition \ref{def:weakSolution} are fulfilled
		for $F:=\{z>0\}$.
		Moreover, $\frak e_0^+$ (see energy inequality \eqref{eqn:EI}) can be chosen to be $\frak E(b^0,z^0)$ which satisfies
		\eqref{eqn:variationalEnergyIneq} by Lemma \ref{lemma:EfrakEestimate}.
	\end{proposition}
	
	\begin{remark}
		Let us consider the functions $e$, $u$ and $z$ obtained above in the degenerate limit.
		We do not know that $F=\{z>0\}$ equals $\frak A_D(\{z>0\})$ and, if $F\setminus\frak A_D(\{z>0\})\neq\emptyset$, it
		is not clear whether $u$ can be extended such that $e=\epsilon(u)$ also holds in $F$.
		On the other hand, we would like to stress that $(u,z^\star)$ with the truncated function $z^\star:=z\mathds 1_{\frak A_D(\{z>0\})}$ also do not
		necessarily form a weak solution in the sense of Definition \ref{def:weakSolution}.
		Because $z^\star$ viewed as an $SBV^2(0,T;L^2(\Omega))$-function may have jumps
		which needs to be accounted for in the energy inequality \eqref{eqn:EI}.
		The construction of weak solutions will be performed in Section \ref{section:mainProb}.
	\end{remark}

		Let $(b^0, z_\varepsilon^0)\rightarrow (b^0,z^0)$
		with $z_\varepsilon^0:=(z-\delta_\varepsilon)^+$ and $b^0:=b(0)$ be a recovery sequence of
		$\frak E_\varepsilon \xrightarrow{\Gamma}\frak E$ according to Lemma \ref{cor:gammaLimitRepresentation} (ii).
		A modification of the proof of Theorem 4.6 in \cite{WIAS1520} yields the following result.
		\begin{theorem}[$\varepsilon$-regularized problem - incomplete damage]
		\label{theorem:regProblem}
			Let $\varepsilon>0$.
			For the given initial-boun\-dary
			data $z_\varepsilon^0\in W^{1,p}(\Omega)$ and $b\in W^{1,1}(0,T;W^{1,\infty}(\Omega;\R^n))$
			there exists a pair $q_\varepsilon=(u_\varepsilon,z_\varepsilon)$ such that
			\begin{enumerate}
				\renewcommand{\labelenumi}{(\roman{enumi})}
				\item
					\textit{Trajectory spaces:}
					\begin{align*}
					\begin{aligned}
						&z_\varepsilon\in L^\infty(0,T;W^{1,p}(\Omega))\cap H^1(0,T;L^2(\Omega)),
						&&u_\varepsilon\in L^\infty(0,T;H^1(\Omega;\mathbb R^n)).
					\end{aligned}
					\end{align*}
				\item
				\textit{Quasi-static mechanical equilibrium:}
					\begin{equation}
					\label{eqn:ID3}
						\int_{\Omega} W_{,e}^\varepsilon(\epsilon(u_\varepsilon(t)),z_\varepsilon(t)):\epsilon(\zeta)\,\mathrm dx=0
					\end{equation}
					for a.e. $t\in(0,T)$ and for all $\zeta\in H_D^{1}(\Omega;\mathbb R^n)$. Furthermore, $u_\varepsilon=b$
					on the boundary $D_T$.
				\item
				\textit{Damage one-sided variational inequality: }
					\begin{align}
					\label{eqn:ID4}
						\int_{\Omega}\left(|\nabla z_\varepsilon(t)|^{p-2}\nabla z_\varepsilon(t)\cdot\nabla\zeta
							+W_{,z}^\varepsilon(\epsilon(u_\varepsilon(t)),z_\varepsilon(t))\zeta\right)\dx
						&\geq\int_\Omega \left(\alpha-\beta\partial_t z_\varepsilon(t)-r_\varepsilon(t)\right)\zeta\dx,\\
						z_\varepsilon(t)&\geq 0,\notag\\
						\partial_t z_\varepsilon(t)&\leq 0\notag
					\end{align}
					for a.e. $t\in(0,T)$ and for all $\zeta\in W^{1,p}(\Omega)$ with $\zeta\leq 0$
					where $r_\varepsilon\in L^1(\Omega_T)$ satisfies
					\begin{align*}
						\int_\Omega r_\varepsilon(t)(\xi-z_\varepsilon(t))\dx\leq 0
					\end{align*}
					for a.e. $t\in(0,T)$ and for all $\xi\in W^{1,p}(\Omega)$ with $\xi \geq 0$.
					The initial value is given by $z_\varepsilon(t=0)=z_\varepsilon^0$ in $\ol\Omega$.
			\item
				\textit{Energy inequality: }
					\begin{align}
						& \C E_\varepsilon(\epsilon(u_\varepsilon(t)),z_\varepsilon(t))
							+\int_{\Omega_t}\left(\alpha|\partial_t z_\varepsilon|+\beta |\partial_t z_\varepsilon|^2\right)\dxs\notag\\
					\label{eqn:ID5}
						&\qquad\qquad\leq \C E_\varepsilon(\epsilon(u_\varepsilon^0),z_\varepsilon^0)
							+\int_{\Omega_t} W_{,e}^\varepsilon(\epsilon(u_\varepsilon),z_\varepsilon):\epsilon(\partial_t b)\dxs
					\end{align}
					holds for a.e. $t\in(0,T)$ where $u_\varepsilon^0$ minimizes $ \C E_\varepsilon(\epsilon(\cdot),z_\varepsilon^0)$ in $H^1(\Omega;\R^n)$
					with Dirichlet data $b^0$ on $D$.
				\end{enumerate}
			Moreover, $r_\varepsilon$ in (iv) can be chosen to be
			\begin{align}
			\label{eqn:rstructure}
				r_\varepsilon=-\chi_\varepsilon W_{,z}(\epsilon(u_\varepsilon),z_\varepsilon)
			\end{align}
			with $\chi_\varepsilon\in L^\infty(\Omega)$ fulfilling $\chi_\varepsilon=0$ on $\{z_\varepsilon>0\}$ and
			$0\leq\chi_\varepsilon\leq 1$ on $\{z_\varepsilon=0\}$.
		\end{theorem}

		We consider a sequence $\{\varepsilon_M\}_{M\in\N}\subseteq(0,1)$ with $\varepsilon_M\rightarrow 0^+$ as $M\rightarrow \infty$
		and for every $M\in\N$ a weak solution $(u_{\varepsilon_M}, z_{\varepsilon_M})$
		of the incomplete damage problem according to Theorem \ref{theorem:regProblem}.
		The index $M$ is omitted in the following.
		We agree that $e_\varepsilon:=\epsilon(u_\varepsilon)$ denotes the strain
		of the regularized system.
		Our further analysis makes also use of the truncated strain $\widehat e_\varepsilon$ (the strain on the not completely damaged parts of $\Omega$)
		given by
		\begin{align*}
			\widehat e_\varepsilon:=e_\varepsilon\mathds{1}_{\{z_\varepsilon>0\}}.
		\end{align*}
		We proceed by deriving suitable a-priori estimates for the incomplete damage problem with respect to $\varepsilon$.
		\begin{lemma}[A-priori estimates]
		\label{lemma:aprioriEstimatesI}
			There exists a $C>0$ independent of $\varepsilon$ such that\\\\
			\begin{tabular}{ll}
				\begin{minipage}{20.0em}
					\begin{enumerate}
						\renewcommand{\labelenumi}{(\roman{enumi})}
						\item
							$\|\widehat e_\varepsilon\|_{L^2(\Omega_T;\mathbb R^{n\times n})}\leq C$,
						\item
							$\sup_{t\in[0,T]}\|z_\varepsilon(t)\|_{W^{1,p}(\Omega)}\leq C$,
					\end{enumerate}
				\end{minipage}
				&
				\begin{minipage}{25em}
					\begin{enumerate}
						\renewcommand{\labelenumi}{(\roman{enumi})}
						\item[(iii)]
							$\|\partial_t z_\varepsilon\|_{L^2(\Omega_T)}\leq C$,
						\item[(iv)]
							$\|W^\varepsilon(e_\varepsilon,z_\varepsilon)\|_{L^\infty(0,T;L^1(\Omega))}\leq C$.
					\end{enumerate}
				\end{minipage}
			\end{tabular}
		\end{lemma}
		\begin{proof}
			Applying Gronwall's lemma to the energy estimate \eqref{eqn:ID5} and noticing the boundedness of
			$\C E_\varepsilon(\epsilon(u_\varepsilon^0),z_\varepsilon^0)$ with respect to
			$\varepsilon\in(0,1)$ show (iii) and
			\begin{align}
			\label{eqn:EnergyBoundedness}
				& \C E_\varepsilon(e_\varepsilon(t),z_\varepsilon(t))\leq C
			\end{align}
			for a.e. $t\in(0,T)$ and all $\varepsilon\in(0,1)$ (cf. \cite{WIAS1520})
			and in particular (iv).
			Taking the restriction $0\leq z_\varepsilon\leq 1$
			into account, property \eqref{eqn:EnergyBoundedness} gives rise to
			$\|z_\varepsilon\|_{L^\infty(0,T;W^{1,p}(\Omega))}\leq C$.
			Together with the control of the time-derivative (iii), we obtain boundedness of $\|z_\varepsilon(t)\|_{W^{1,p}(\Omega)}\leq C$ for
			every $t\in[0,T]$ and $\varepsilon\in(0,1)$. Hence, (ii) is proven.
			
			It remains to show (i). To proceed, we test inequality \eqref{eqn:ID4} with $\zeta\equiv-1$ and integrate from $t=0$ to $t=T$:
			\begin{align}
			\label{eqn:VIestimate1}
				\int_{\Omega_T}W_{,z}^\varepsilon(e_\varepsilon,z_\varepsilon)
				+r_\varepsilon\dxt
				\leq\int_{\Omega_T}\left(\alpha-\beta\,\partial_t z_\varepsilon\right)\dxt.
			\end{align}
			Applying \eqref{eqn:assumptiong}, \eqref{eqn:rstructure} and \eqref{eqn:VIestimate1}, yield
			\begin{align*}
				\int_{\Omega_T}\eta|\widehat e_\varepsilon|^2\dxt
				&=\int_{\{z_\varepsilon>0\}}\eta|e_\varepsilon|^2\dxt\\
				&\leq C\int_{\{z_\varepsilon>0\}}\frac12 g'(z_\varepsilon)\mathbb C e_\varepsilon:e_\varepsilon\dxt\\
				&=C\int_{\Omega_T}
					W_{,z}^\varepsilon(e_\varepsilon,z_\varepsilon)\dxt-\int_{\{z_\varepsilon=0\}}W_{,z}^\varepsilon(e_\varepsilon,z_\varepsilon)\dxt\\
				&\leq C\int_{\Omega_T}
					W_{,z}^\varepsilon(e_\varepsilon,z_\varepsilon)\dxt-\int_{\Omega_T}\chi_\varepsilon W_{,z}^\varepsilon(e_\varepsilon,z_\varepsilon)\dxt\\
				&=C\int_{\Omega_T}\left(W_{,z}^\varepsilon(e_\varepsilon,z_\varepsilon)+r_\varepsilon(t)\right)\dxt\\
				&\leq C\int_{\Omega_T}\left(\alpha-\beta\,\partial_t z_\varepsilon\right)\dxt.
			\end{align*}
			This and the boundedness of $\int_{\Omega_T}\left(\alpha-\beta\,\partial_t z_\varepsilon\right)\dxt$ with respect to $\varepsilon$ shows (i).
			\ep
		\end{proof}
		
			\begin{lemma}[Converging subsequences]
			\label{lemma:convergingSubsequences}
				There exists functions
				\begin{align*}
				\begin{aligned}
					&\widehat e\in L^2(\Omega_T;\mathbb R^{n\times n}),
					&&z\in L^\infty(0,T;W^{1,p}(\Omega))\cap H^1(0,T;L^2(\Omega)),
				\end{aligned}
				\end{align*}
				where $z$ is monotonically decreasing with respect to $t$, i.e. $\partial_t z\leq 0$,
				and a subsequence (we omit the index) such that for $\varepsilon\rightarrow 0^+$\\
				
				\begin{tabular}{ll}
					\hspace{-2em}
					\begin{minipage}{20em}
						{\begin{enumerate}
							\item[(i)]
								$z_\varepsilon\rightharpoonup z\text{ in }H^1(0,T;L^2(\Omega))$,\\
								$z_\varepsilon\rightarrow z\text{ in }L^p(0,T;W^{1,p}(\Omega))$,\\
								$z_\varepsilon(t)\weaklim z(t)\text{ in }W^{1,p}(\Omega)$,\\
								$z_\varepsilon\rightarrow z\text{ in } \ol{\Omega_T}$,
						\end{enumerate}}
					\end{minipage}
					&\hspace{-1.5em}
					\begin{minipage}{26em}
						{\begin{enumerate}
							\item[(ii)]
								$\widehat e_{\varepsilon}\rightharpoonup \widehat e\text{ in }L^2(\Omega_T;\mathbb R^{n\times n})$,\\
								$W_{,e}^\varepsilon(e_\varepsilon,z_\varepsilon)\weaklim W_{,e}(\widehat e,z)$
								in $L^2(\{z>0\};\mathbb R^{n\times n})$,\\
								$W_{,e}^\varepsilon(e_\varepsilon,z_\varepsilon)\rightarrow 0$
								in $L^2(\{z=0\};\mathbb R^{n\times n})$.\\\vspace*{0.3em}
						\end{enumerate}}
					\end{minipage}
				\end{tabular}
			\end{lemma}
			\begin{proof}
				The a-priori estimates from Lemma \ref{lemma:aprioriEstimatesI} and classical compactness theorems as well as
				compactness theorems from J.-L. Lions and T. Aubin yield \cite{Simon}
				
				\begin{tabular}{ll}
						\hspace*{-4.2em}
					\begin{minipage}{21.7em}
						{\begin{align*}
							&z_\varepsilon\stackrel{\star}{\rightharpoonup} z\text{ in }L^\infty(0,T;W^{1,p}(\Omega)),\\
							&z_\varepsilon\rightharpoonup z\text{ in }H^1(0,T;L^2(\Omega)),\\
							&z_\varepsilon\rightarrow z\text{ in }L^p(\Omega_T)
						\end{align*}}
					\end{minipage}
					&
					\begin{minipage}{20em}
						{\begin{align*}
							&\widehat e_{\varepsilon}\rightharpoonup \widehat e\text{ in }L^2(\Omega_T;\mathbb R^{n\times n}),\\
							&W_{,e}^\varepsilon(e_\varepsilon,z_\varepsilon)\weaklim w_e\text{ in }
								L^2(\Omega_T;\mathbb R^{n\times n}),\\
						\end{align*}}
					\end{minipage}
				\end{tabular}\\
				as $\varepsilon\rightarrow 0^+$ for a subsequence and appropriate functions $w_e$, $\widehat e$ and $z$.

				Proving the strong convergence of $\nabla z_\varepsilon$ in $L^p(\Omega_T;\R^n)$ does not substantially differ
				from the proof presented in \cite{WIAS1520}.
				It is essentially based on the elementary inequality
				\begin{equation*}
					C_\mathrm{uc}|x-y|^p \le \left\langle\big(|x|^{p-2} x - |y|^{p-2}y \big),
					x-y\right\rangle,
				\end{equation*}
				where $\langle\cdot,\cdot\rangle$ denotes the standard Euclidean scalar product
				and on an approximation scheme $\{\zeta_\varepsilon\}\subseteq L^p(0,T;W^{1,p}(\Omega))$ with $\zeta_\varepsilon\geq 0$ and
				\begin{subequations}
					\begin{align}
						\label{eqn:approxConvergence}
						&\zeta_\varepsilon\rightarrow z\text{ in }L^p(0,T;W^{1,p}(\Omega))\text{ as }\varepsilon\rightarrow 0^+,\\
						\label{eqn:approxRestriction}
						&0\leq \zeta_\varepsilon\leq z_\varepsilon\text{ a.e. in }\Omega_T\text{ for all }\varepsilon\in(0,1).
					\end{align}
				\end{subequations}
				Using the above properties, we obtain the estimate:
				\begin{align*}
					C_\mathrm{uc}\int_{\Omega_T}|\nabla z_\varepsilon-\nabla z|^p\dxt
					&\leq\int_{\Omega_T}(|\nabla z_\varepsilon|^{p-2}\nabla z_\varepsilon
						-|\nabla z|^{p-2}\nabla z)\cdot \nabla(z_\varepsilon-z)\dxt\\
					&=\underbrace{\int_{\Omega_T}|\nabla z_\varepsilon|^{p-2}\nabla z_\varepsilon
						\cdot \nabla(z_\varepsilon-\zeta_\varepsilon)\dxt}_{A_\varepsilon}\\
					&\quad+\underbrace{\int_{\Omega_T}|\nabla z_\varepsilon|^{p-2}\nabla z_\varepsilon
						\cdot \nabla(\zeta_\varepsilon-z)-|\nabla z|^{p-2}\nabla z
						\cdot \nabla(z_\varepsilon-z)\dxt}_{B_\varepsilon}.
				\end{align*}
				The weak convergence property of $\{\nabla z_\varepsilon\}$ in $L^p(\Omega_T)$ and \eqref{eqn:approxConvergence}
				show $B_\varepsilon\rightarrow 0$ as $\varepsilon\rightarrow 0^+$.
				Property \eqref{eqn:ID4} tested with $\zeta(t)=\zeta_\varepsilon(t)-z_\varepsilon(t)$
				and integration from $t=0$ to $t=T$
				yields
				\begin{align*}
					A_\varepsilon\leq{}&
						\underbrace{\int_{\Omega_T}W_{,z}^\varepsilon(\epsilon(u_\varepsilon),z_\varepsilon)
							(\zeta_\varepsilon-z_\varepsilon)\dxt}_{\leq 0\text{ by } \eqref{eqn:assumptiong}\text{ and }\eqref{eqn:approxRestriction}}
							+\underbrace{\int_{\Omega_T}\left(-\alpha+\beta(\partial_t z_\varepsilon(t))\right)(\zeta_\varepsilon-z_\varepsilon)
							\dxt}_{\rightarrow 0\text{ as }\varepsilon\rightarrow 0^+\text{ by }\eqref{eqn:approxConvergence}}.
				\end{align*}
				Here, we have used $r_\varepsilon\zeta=0$ on $\Omega_T$ (see \eqref{eqn:rstructure}).
				Therefore, (i) is also shown.
				
				To prove (ii), we define $N_\varepsilon$ to be $\{z_\varepsilon>0\}\cap\{z>0\}$.
				Consequently, we get
				\begin{align}
					\label{eqn:We}
					&W_{,e}^\varepsilon(\widehat e_\varepsilon,z_\varepsilon)\mathds{1}_{N_\varepsilon}
						=W_{,e}^\varepsilon(e_\varepsilon,z_\varepsilon)\mathds{1}_{N_\varepsilon}
				\end{align}
				and the convergence
				\begin{align}
				\label{eqn:domainConvergence}
					\mathds{1}_{N_\varepsilon}\rightarrow \mathds{1}_{\{z>0\}}\text{ in }\Omega_T
				\end{align}
				for $\varepsilon\rightarrow 0^+$ by using $z_\varepsilon\rightarrow z$ in $\ol{\Omega_T}$.
				Calculating the weak $L^1(\Omega_T;\mathbb R^{n\times n})$-limits in
				\eqref{eqn:We} for $\varepsilon\rightarrow 0^+$ on both sides
				by using the already proven convergence properties,
				we obtain $W_{,e}(\widehat e,z)=w_e$.
				The remaining convergence property in (ii) follow from Lemma \ref{lemma:aprioriEstimatesI} (iv).
				\ep
			\end{proof}
			
			We now introduce the shrinking set $F\subseteq\ol{\Omega_T}$ by defining
			\begin{align*}
				&F(t):=\{z(t)>0\}
			\end{align*}
			for all $t\in[0,T]$. This is a well-defined object since $F\subseteq\ol{\Omega_T}$ is relatively open by Theorem \ref{theorem:continuity} as well as
			$F(s)\subseteq F(t)$ for all $0\leq t\leq s\leq T$ by the monotone decrease of $z(x,\cdot)$.

			\begin{corollary}
			\label{cor:inclusion}
				Let $t\in[0,T]$ and $U\compact F(t)$ be an open subset.
				Then $U\subseteq \{z_\varepsilon(s)>0\}$ for all $s\in[0,t]$ provided that $\varepsilon>0$ is sufficiently small.
				More precisely, there exist $0<\varepsilon_0,\eta<1$ such that
				\begin{align*}
					z_\varepsilon(s)\geq \eta\text{ in }U
				\end{align*}
				for all $s\in[0,t]$ and for all $0<\varepsilon<\varepsilon_0$.
			\end{corollary}
			\begin{proof}
				By assumption, we obtain the property $\mathrm{dist}(U,\{z(t)=0\})>0$.
				Therefore, and by $z(t)\in\C C(\ol\Omega)$, we find an $\eta>0$ such that $z(t)\geq 2\eta$ in $U$.
				By exploiting the convergence $z_\varepsilon(t)\rightarrow z(t)$ in $\C C(\overline{\Omega})$ as $\varepsilon\rightarrow 0^+$ by Lemma
				\ref{lemma:convergingSubsequences} (b) and the compact embedding $W^{1,p}(\Omega)\hookrightarrow \C C(\ol\Omega)$,
				there exists an $\varepsilon_0>0$ such that $z_\varepsilon(t)\geq \eta$ on $U$ for all $0<\varepsilon<\varepsilon_0$.
				Finally, the claim follows from the fact that $z_\varepsilon$ is monotonically decreasing with respect to $t$.
				\ep
			\end{proof}
		
		\begin{lemma}
		\label{lemma:deformationField}
			There exists a function $u\in L_t^2H^1_{x,\mathrm{loc}}(\frak A_D(F);\R^n)$ such that
			\begin{enumerate}
				\renewcommand{\labelenumi}{(\roman{enumi})}
				\item
					$\epsilon(u)=\widehat e$ a.e. in $\frak A_D(F)$,
				\item
					$u=b$ on the boundary $D_T\cap \frak A_D(F)$.
			\end{enumerate}
		\end{lemma}
		\begin{proof}
				Let $\{U_k^m\}$ and $\{t_m\}$ be sequences satisfying the properties of Corollary \ref{cor:representationFormula} applied to $\frak A_D(F)$.
				We get for each fixed $k,m\in\N$
				\begin{align}
				\label{eqn:Ginclusion}
					U_k^m\times [0,t_m]\subseteq \{z_\varepsilon>0\}
				\end{align}
				for all $0<\varepsilon\ll1$ due to Corollary \ref{cor:inclusion}.
				Inclusion \eqref{eqn:Ginclusion} implies
				\begin{align}
				\label{eqn:eueFrakEquality}
					\epsilon(u_\varepsilon)=\widehat e_{\varepsilon}
				\end{align}
				a.e. in $U_k^m\times(0,t_m)$.
				Korn's inequality applied on the Lipschitz domain $U_k^m$ yields
				(note that $\C H^{n-1}(\partial U_k^m\cap D)>0$)
				\begin{align*}
					\|u_{\varepsilon}\|_{L^2(0,t_m;H^1(U_k^m;\mathbb R^{n}))}^2
						&\leq 2\int_0^{t_m}\|u_{\varepsilon}(t)-b(t)\|_{H^1(U_k^m;\mathbb R^{n})}^2
						+\|b(t)\|_{H^1(U_k^m;\mathbb R^{n})}^2\,\mathrm dt\\
					&\leq C\left(1+\int_0^{t_m}\|\epsilon(u_{\varepsilon}(t))\|_{L^2(U_k^m;\mathbb R^{n\times n})}^2\,\mathrm dt\right)\\
					&\leq C\left(1+\int_0^{t_m}\|\widehat e_{\varepsilon}(t))\|_{L^2(\Omega;\mathbb R^{n\times n})}^2\,\mathrm dt\right).
				\end{align*}
				with a constant $C=C(U_k^m,b)>0$.
				Together with the boundedness of $\widehat e_{\varepsilon}$ in $L^2(\Omega_T;\R^{n\times n})$, we can find a subsequence
				$\varepsilon\rightarrow 0^+$ and a function $u^{(k,m)}\in L^2(0,t_m;H^1(U_k^m;\mathbb R^{n}))$
				such that
				\begin{align}
				\label{eqn:uConvergence}
					u_{\varepsilon}\rightharpoonup u^{(k,m)}\text{ in }L^2(0,t_m;H^1(U_k^m;\mathbb R^{n})).
				\end{align}
				Thus $\epsilon(u^{(k,m)})=\widehat e$ in $U_k^m\times(0,t_m)$ because of \eqref{eqn:eueFrakEquality} and the
				weak convergence property of $\widehat e_\varepsilon$.
				For each $k,m\in\N$, we can apply the argumentation above.
				Therefore, by successively choosing subsequences and by applying a diagonalization
				argument, we obtain a subsequence $\varepsilon\rightarrow 0^+$ such that \eqref{eqn:uConvergence} holds for all $k,m\in\N$.
				
				Since $u^{(k_1,m_1)}=u^{(k_2,m_2)}$ a.e. on $U_{k_1}^{m_1}\times(0,t_{m_1})\cap U_{k_2}^{m_2}\times(0,t_{m_2})$ for all $k_1,k_2,m_1,m_2\in\N$,
				we obtain an $u:\frak A_D(F)\rightarrow\R^n$ such that $u|_{U_k^m\times(0,t_m)}\in L^2(0,t_m;H^1(U_k^m;\R^n))$ for all $m\in\N$.
				Proposition \ref{prop:localSob} (a) yields $u\in L_t^2H^1_{x,\mathrm{loc}}(F;\R^n)$ and the symmetric
				gradient $\epsilon(u)$ coincides with $\widehat e$. Therefore, (i) is shown.
				
				Furthermore, for every $k,m\in\N$, we have $u(t)=b(t)$ on $\partial U_k^m\cap D$ in the sense of traces for a.e. $t\in[0,t_m]$.
				By Proposition \ref{prop:localSob} (b), (ii) follows.
				\ep
		\end{proof}\\
		We are now able to prove Proposition \ref{prop:existenceResultSimplified}.\\\\
		\begin{proof}[Proof of Proposition \ref{prop:existenceResultSimplified}]
			Lemma \ref{lemma:convergingSubsequences} and Lemma \ref{lemma:deformationField} give the desired regularity properties of the functions
			$(e,u,z)$ in Proposition \ref{prop:existenceResultSimplified}.
			Here, we set $e:=\widehat e|_F\in L^2(F;\R^{n\times n})$.
			The property $e=\epsilon(u)$ in $\frak A_D(F)$ follows from Lemma \ref{lemma:deformationField}.
			
			In the following, we are going to prove that properties (ii)-(v)
			of Definition \ref{def:weakSolution} are satisfied.
			\begin{enumerate}
				\renewcommand{\labelenumi}{(\roman{enumi})}
				\item[(ii)]
					Lemma \ref{lemma:convergingSubsequences} (ii) allows us to pass to $\varepsilon\rightarrow0^+$ in \eqref{eqn:ID3} integrated from $t=0$ to $t=T$.
					Therefore, equation \eqref{eqn:forceBalanceWeak} holds for a.e. $t\in(0,T)$ and all $\zeta\in H_{D}^1(\Omega;\R^n)$.
					Moreover, the boundary condition $u=b$ on $D_T\cap \frak A_D(F)$ is satisfied. Definition \ref{def:admissibleSubset} immediately
					implies $D_T\cap F=D_T\cap \frak A_D(F)$.
				\item[(iii)]
					We first show \eqref{eqn:VI}.
					Let $\zeta\in L^\infty(0,T;W^{1,p}(\Omega))$ with $\zeta\leq 0$.
					The variational inequality \eqref{eqn:ID4} and the representation for $r_\varepsilon$ \eqref{eqn:rstructure} imply
					\begin{align}
					\label{eqn:VIapprox}
						&0\leq \int_{\Omega_T}\left(|\nabla z_\varepsilon|^{p-2}\nabla z_\varepsilon\cdot\nabla\zeta+(-\alpha+\beta\partial_t z_\varepsilon)\zeta\right)\dxt
							+\int_{\{z_\varepsilon>0\}}W_{,z}^\varepsilon(e_\varepsilon,z_\varepsilon)\zeta\dxt.
					\end{align}
					In addition,
					\begin{align*}
						\int_{\{z_\varepsilon>0\}}W_{,z}^\varepsilon(e_\varepsilon,z_\varepsilon)\zeta\dxt
							\leq{}&\int_{F\cap\{z_\varepsilon>0\}}W_{,z}^\varepsilon(e_\varepsilon,z_\varepsilon)\zeta\dxt\\
							={}&\int_{F}g'(z_\varepsilon)\mathbb C \widehat e_\varepsilon:\widehat e_\varepsilon\zeta\dxt
					\end{align*}
					Lemma \ref{lemma:convergingSubsequences}, a lower semi-continuity argument and
					$\mathds{1}_{\{z_\varepsilon>0\}\cap \{z=0\}}\rightarrow \mathds{1}_{\{z=0\}}$ a.e. in $\Omega_T$
					(see proof of Lemma \ref{lemma:convergingSubsequences}) yield
					\begin{align*}
						\limsup_{\varepsilon\rightarrow 0^+}\int_{\{z_\varepsilon>0\}}W_{,z}^\varepsilon(e_\varepsilon,z_\varepsilon)\zeta\dxt
							&\leq\int_{F}W_{,z}^\varepsilon(e,z)\zeta\dxt.
					\end{align*}
					Therefore, applying ''$\limsup_{\varepsilon\rightarrow 0^+}$`` on both sides of \eqref{eqn:VIapprox}, using the above estimate and
					Lemma \ref{lemma:convergingSubsequences} yield
					\begin{align}
					\label{eqn:VItemp}
						&\int_{F}|\nabla z|^{p-2}\nabla z\cdot\nabla\zeta+W_{,z}(e,z)\zeta\dx
						\geq\int_\Omega(\alpha-\beta\partial_t z)\zeta\dx.
					\end{align}
					The properties $\partial_t z\leq 0$ and $z\geq 0$ a.e. in $\Omega_T$ follow from
					Lemma \ref{lemma:convergingSubsequences} by taking $\partial_t z_\varepsilon\leq 0$ and $z_\varepsilon\geq 0$ a.e. in $\Omega_T$
					into account.
				\item[(iv)]
					The jump condition \eqref{eqn:damageJumpCondition} in (iv) of Definition \ref{def:weakSolution} holds trivially since
					we have the regularity
					$z\in L^\infty(0,T;W^{1,p}(\Omega))\cap H^1(0,T;L^2(\Omega))$.
				\item[(v)]
					To complete the proof, we need to show the energy estimates \eqref{eqn:EI}.
					Since $\{b^0,z_\varepsilon^0\}$ is a recovery sequence, we get
					$\C E_\varepsilon(\epsilon(u_\varepsilon^0),z_\varepsilon^0)\rightarrow \frak E(b^0,z^0)$
					as $\varepsilon\rightarrow 0^+$.
					Now, applying ''$\limsup_{\varepsilon\rightarrow0^+}$`` on both sides in \eqref{eqn:ID5} and using the convergence properties
					in Lemma \ref{lemma:convergingSubsequences} as well as lower semi-continuity arguments yield
					\begin{align}
						&\frak E(b^0,z^0)+\int_0^t\int_{F(s)} W_{,e}(e,z):\epsilon(\partial_t b)\dxs\notag\\
						&\qquad\geq \limsup_{\varepsilon\rightarrow 0^+}\left(\C E_\varepsilon(e_\varepsilon(t),z_\varepsilon(t))
							+\int_{\Omega_t} \alpha|\partial_t z_\varepsilon|+\beta |\partial_t z_\varepsilon|^2\dxs\right)\notag\\
						&\qquad\geq
							\limsup_{\varepsilon\rightarrow 0^+}\int_\Omega W^\varepsilon(e_\varepsilon(t),z_\varepsilon(t))\dx
							+\int_{\Omega}\frac 1p|\nabla z(t)|^p\dx\notag\\
						&\qquad\quad+\int_{\Omega_t}\alpha|\partial_t z|+\beta |\partial_t z|^2\dxs.
					\label{eqn:EnergyInequalityTemp}
					\end{align}
					Indeed, for an arbitrary $t\in(0,T)$, we derive by Fatou's lemma and Lemma \ref{lemma:convergingSubsequences}
					\begin{align}
						\int_0^t\left(\limsup_{\varepsilon\rightarrow 0^+}\int_\Omega W^\varepsilon(e_\varepsilon(s),z_\varepsilon(s))\dx\right)\ds\notag
						&\geq
							\limsup_{\varepsilon\rightarrow 0^+}\int_{\Omega_t}W^\varepsilon(e_\varepsilon,z_\varepsilon)\dxs\notag\\
						&\geq\liminf_{\varepsilon\rightarrow 0^+}\int_{F}(g(z_\varepsilon)+\varepsilon)
							\mathbb C\widehat e_\varepsilon:\widehat e_\varepsilon\dxs\notag\\
						&\geq\int_{F}W(e,z)\dxs.
					\label{eqn:energyLowerSemicontinuity}
					\end{align}
					We have used the weak convergence property
					$$
						\sqrt{g(z_\varepsilon)+\varepsilon}\;\widehat e_\varepsilon\weaklim \sqrt{g(z)}\;\widehat e\text{ in }L^2(\Omega_T;\R^{n\times n})
					$$
					as $\varepsilon\rightarrow 0^+$.
					To the end, \eqref{eqn:energyLowerSemicontinuity} implies
					$$
						\limsup_{\varepsilon\rightarrow 0^+}\int_\Omega W^\varepsilon(e_\varepsilon(t),z_\varepsilon(t))\dx\geq
							\int_{F(t)}W(e(t),z(t))\dx
					$$
					for a.e. $t\in(0,T)$.
					Combining it with \eqref{eqn:EnergyInequalityTemp}, estimate \eqref{eqn:EI} is shown.
					\ep
			\end{enumerate}
		\end{proof}

	\subsection{Existence of weak solutions}
	\label{section:mainProb}
		By using the achievements in the previous section and Zorn's lemma, we will prove the main results,
		Theorem \ref{theorem:mainExistenceResult} and Remark \ref{theorem:localExistence}.
		To proceed, let $\eta>0$ be fixed and $\C P$ be the set
		\begin{align*}
			\C P:=\big\{(\widehat T,e,u,z,F)\,|\,&0<\widehat T\leq T\text{ and }(e,u,z,F)
				\text{ is a weak solution on }\\
			&\text{$[0,\widehat T]$ with fineness $\eta$ according to Definition \ref{def:approxWeakSolution}
				}\big\}.
		\end{align*}
		We introduce a partial ordering $\leq$ on $\C P$ by
		\begin{align*}
			(\widehat T_1,e_1,u_1,z_1,F_1)\leq (\widehat T_2,e_2,u_2,z_2,F_2)\quad\Leftrightarrow\quad&
				\widehat T_1\leq \widehat T_2,\,e_2|_{[0,\widehat T_1]}=e_1,\,u_2|_{[0,\widehat T_1]}=u_1,\\
			&z_2|_{[0,\widehat T_1]}=z_1,\,F_2|_{[0,\widehat T_1]}=F_1.
		\end{align*}
		The next two lemma prove the assumptions for Zorn's lemma.
		\begin{lemma}
		\label{lemma:Pnonempty}
			$\C P\neq\emptyset$.
		\end{lemma}
		\begin{proof}
			Let $(e,u,z)$ be the tuple from Proposition \ref{prop:existenceResultSimplified}
			to the initial-boundary data $(z^0,b)$.
			If there exists an $\varepsilon>0$ such that $J_{z^{\star}}\cap[0,\varepsilon]=\emptyset$
			with $z^{\star}(t):=z(t)\mathds 1_{\frak A_D(\{z^-(t)>0\})}$ then
			$(\varepsilon,e,u,z,F)\in\C P$.
			Otherwise, we find $0\in C_{z^{\star}}$.
			We claim
			\begin{align}
			\label{eqn:inclusion1}
				\C L^n\left(\{z^0>0\}\setminus\frak A_D(\{z(t)>0\})\right)\rightarrow 0\text{ as }t\rightarrow 0^+.
			\end{align}
			We consider the non-trivial case $z^0\not\equiv 0$.
			Let $x\in \{z^0>0\}\cap\Omega$.
			Since $\{z^0>0\}\subseteq\ol{\Omega_T}$ is relatively open and admissible with respect to $D$, there exists a Lipschitz domain $U\compact \{z^0>0\}$ with $x\in U$
			such that $\C H^{n-1}(\partial U\cap D)>0$ by Lemma \ref{lemma:LipschitzDomain}.
			Because of Theorem \ref{theorem:continuity}, $z\in\C C(\ol{\Omega_T})$ and, consequently, there exists a $t>0$ such that $U\compact \{z(s)>0\}$ for all $0\leq s<t$.
			In particular, $x\in \frak A_D(\{z(s)>0\})$ for all $0\leq s<t$.
			This proves \eqref{eqn:inclusion1}.
			Finally, choose $\varepsilon>0$ so small such that $\varepsilon<\eta$ and (note the monotonicity of $z$ with respect to $t$)
			$$
				\C L^n\left(\{z(t)>0\}\setminus\frak A_D(\{z(t)>0\})\right)\leq\C L^n\left(\{z^0>0\}\setminus\frak A_D(\{z(t)>0\})\right)<\eta
			$$
			for all $0\leq t<\varepsilon$.
			We have proved that $(e,u,z)$ on $F:=\{z>0\}$ is a weak solution with fineness $\eta$ on the time interval $[0,\varepsilon]$,
			i.e. $(\varepsilon,e,u,z,F)\in\C P$.
			\ep
		\end{proof}
		
		\begin{lemma}
		\label{lemma:PupperBound}
			Every totally ordered subset of $\C P$ has an upper bound.
		\end{lemma}
		\begin{proof}
			Let $\C R\subseteq\C P$ be a totally ordered subset.
			We denote with $[0,T_R]$ the corresponding time interval of an element $R\in \C R$.
			Let us select a sequence $\big\{T_\theta,e_\theta,u_\theta,z_\theta,F_\theta\big\}_{\theta\in(0,1)}\subseteq\C R$,
			with $T_{\theta_1}\leq T_{\theta_2}$ for $\theta_2\leq \theta_1$ and
			$\lim_{\theta\rightarrow 0^+}T_\theta=\sup_{Q\in\C R}T_Q=:\widehat T$.
			
			Let $t\in(0,\widehat T)$. There exists a $\theta\in(0,1)$ with $T_\theta\geq t$ and we define
			\begin{align*}
				(e(t),u(t),z(t),F(t))
				:=(e_\theta(t),u_\theta(t),z_\theta(t),F_\theta(t)).
			\end{align*}
			By construction, the functions $(e,u,z)$ satisfy
			the properties (ii)-(v) of Definition \ref{def:weakSolution} on $[0,\widehat T]$.
			It remains to show that $(e(t),u(t),z(t))$ are in the trajectory spaces as in Definition \ref{def:approxWeakSolution} (i)
			and that $F$ satisfies Definition \ref{def:approxWeakSolution} (ii).
			
			The energy estimate for $(e_\theta,u_\theta,z_\theta)$ implies
			\begin{align}
				&\C E (e(t),z(t))+\int_0^t\int_{F(s)}\left(\alpha|\partial_t^\mathrm{a} z|+\beta|\partial_t^\mathrm{a} z|^2\right)\dxs\notag\\
				&\qquad\leq \frak e_0^++\int_0^t\int_{F(s)}W_{,e}(e,z):\epsilon(\partial_t b)\dxs
			\label{eqn:energyEstimateTemp}
			\end{align}
			for a.e. $t\in(0,\widehat T)$.
			Gronwall's lemma yields boundedness of the left hand side of \eqref{eqn:energyEstimateTemp} with respect to a.e. $t\in(0,\widehat T)$.
			
			We immediately get
			\begin{align}
			\label{eqn:zSpace}
				&z\in L^\infty(0,\widehat T;W^{1,p}(\Omega))\cap SBV^2(0,\widehat T;L^2(\Omega)),
			\end{align}
			Variational inequality \eqref{eqn:VI} tested with $\zeta\equiv-1$ shows
			\begin{align*}
				\int_{F(t)}W_{,z}(e(t),z(t))\dx\leq \int_\Omega\alpha\dx-\int_{F(t)}\beta\partial_t^\mathrm{a} z(t)\dx
			\end{align*}
			for a.e. $t\in(0,\widehat T)$. This implies
			\begin{align}
			\label{eqn:eSpace}
				e\in L^2(F;\R^{n\times n}).
			\end{align}
			We know that $u|_{U\times(0,t)}\in L^2(0,t;H^1(U;\R^n))$ for all $t\in(0,\widehat T)$ and all open subsets
			$U\compact \frak A_D(F(t))$.
			Let $\{U_k\}$ be a Lipschitz cover of the admissible set
			$$
				F(\widehat T):=\frak A_D(\{z^-(\widehat T)>0\})
			$$
			according to Lemma \ref{lemma:LipschitzDomain} (in particular, Definition \ref{def:approxWeakSolution} (ii) is fulfilled).
			For each $k\in\N$, we apply Korn's inequality and get for all $t\in(0,\widehat T)$
			\begin{align*}
				\|u-b\|_{L^2(0,t;H^1(U_k;\R^n))}\leq C\|\epsilon(u)\|_{L^2(0,t;L^2(U_k;\R^n))},
			\end{align*}
			where $C>0$ depends on the domain $U_k$ but not on the time $t$.
			Thus $u|_{U_k\times(0,\widehat T)}\in L^2(0,T;H^1(U_k;\R^n))$.
			In conclusion,
			\begin{align}
			\label{eqn:uSpace}
				u\in L_t^2H^1_{x,\mathrm{loc}}(F;\R^n).
			\end{align}
			
			Therefore, property (i) of Definition \ref{def:approxWeakSolution} follows by \eqref{eqn:zSpace}-\eqref{eqn:uSpace}.
			We end up with\linebreak
			$\{\widehat T,e,u,z,F\}\in\C P$ satisfying
			$\{T_\theta,e_\theta,u_\theta,z_\theta,F_\theta\}\leq \{\widehat T,e,u,z,F\}$
			for all $\theta\in(0,1)$.
		\ep
		\end{proof}\\
		Weak solutions exhibit the following concatenation property.
		\begin{lemma}
		\label{lemma:concatenation}
			Let $t_1<t_2<t_3$ be real numbers.
			Suppose that
			\begin{align*}
				&\text{$\widetilde q:=(\widetilde e,\widetilde u,\widetilde z,\widetilde F)$ is a weak solution on $[t_1,t_2]$},\\
				&\text{$\widehat q:=(\widehat e,\widehat u,\widehat z,\widehat F)$ is a weak solution on $[t_2,t_3]$}\\
				&\qquad\qquad\qquad\;\;\;\text{with $\widehat{\frak e}_{t_2}^+=\frak E(\widehat b(t_2),\widehat z^+(t_2))$
				(the value $\frak e_{t_2}^+$ for $\widehat q$ in Definition \ref{def:weakSolution}).}
			\end{align*}
			Furthermore, suppose
			the compatibility condition $\widehat z^+(t_2)=\widetilde z^-(t_2)\mathds 1_{\frak A_D(\{\widetilde z^-(t_2)>0\})}$
			and the Dirichlet boundary data $b\in W^{1,1}(t_1,t_3;W^{1,\infty}(\Omega;\R^n))$.
			Then, we obtain that $q:=(e,u,z,F)$ defined as
			$q|_{[t_1,t_2)}:=\widetilde q$ and
			$q|_{[t_2,t_3]}:=\widehat q$
			is a weak solution on $[t_1,t_3]$.
		\end{lemma}
		\begin{proof}
			Applying ``$\lim_{s\rightarrow t_2^-}\essinf_{\tau\in(s,t_2)}$'' on both sides of the energy estimate \eqref{eqn:EI} for
			$(\widetilde e,\widetilde u,\widetilde z,\widetilde F)$ yields
			\begin{align*}
			\begin{split}
				&\lim_{s\rightarrow t_2^-}\mathop\mathrm{ess\,inf}_{\tau\in(s,t_2)}\C E (e(\tau),z(\tau))
					+\int_{t_1}^{t_2}\int_{F(s)}\left(\alpha|\partial_t^\mathrm{a} z|+\beta|\partial_t^\mathrm{a} z|^2\right)\dxs
					+\liminf_{s\rightarrow t_2^-}\sum_{\tau\in J_{z}\cap(t_1,s]}\left(\C E^-(s)-\C E^+(s)\right)\\
				&\qquad\qquad\qquad\leq \frak e_{t_1}^++\int_{t_1}^{t_2}\int_{F(s)}W_{,e}(e,z):\epsilon(\partial_t b)\dxs.
			\end{split}
			\end{align*}
			This estimate can be rewritten as
			\begin{align}
				&\frak E(b(t_2),z^+(t_2))
					+\int_{t_1}^{t_2}\int_{F(s)}\left(\alpha|\partial_t^\mathrm{a} z|+\beta|\partial_t^\mathrm{a} z|^2\right)\dxs\notag\\
				&+\liminf_{s\rightarrow t_2^-}\sum_{\tau\in J_{z}\cap(t_2,s]}\left(\C E^-(s)-\C E^+(s)\right)
					+\lim_{s\rightarrow t_2^-}\mathop\mathrm{ess\,inf}_{\tau\in(s,t_2)}\C E(e(\tau),z(\tau))
					-\frak E(b(t_2),z^+(t_2))\notag\\
				&\qquad\qquad\qquad
					\leq \frak e_{t_1}^++\int_{t_1}^{t_2}\int_{F(s)}W_{,e}(e,z):\epsilon(\partial_t b)\dxs.
				\label{eqn:EIextended}
			\end{align}
			
			In the following, we show that we may choose the value $\frak E(b(t_2),z^+(t_2))$ for $\frak e_{t_2}^+$.
			By the property (i) of Definition \ref{def:weakSolution}, we get
			$z^-(s)\weaklim z^-(t_2)$ in $W^{1,p}(\Omega)$ and $b(s)\rightarrow b(t_2)$ in $W^{1,\infty}(\Omega;\R^n)$
			as $s\rightarrow t_2^-$.
			In particular, by using Lemma \ref{lemma:truncation} and the monotone decrease of $z^-$ with respect to $t$,
			$$
				z^-(s)\mathds 1_{\frak A_D(\{z^-(s)>0\})}
				\weaklim
				z^-(t_2)\mathds 1_{\bigcap_{\tau\in(t_1,t_2)}\frak A_D(\{ z^-(\tau)>0\})}=:\chi
			$$
			in $W^{1,p}(\Omega)$ as $s\rightarrow t_2^-$.
			By the definition of $\chi$, the inclusion
			$$
				\frak A_D(\{z^-(t_2)>0\})\subseteq\bigcap_{\tau\in(t_1,t_2)}\frak A_D(\{z^-(\tau)>0\})
			$$
			and the compatibility condition, we find $z^+(t_2)=\chi\mathds 1_{\frak A_D(\{z^-(t_2)>0\})}$.
			
			Thus, applying Lemma \ref{lemma:EfrakEestimate}, lower semi-continuity of the $\Gamma$-limit $\frak E$ and Corollary \ref{cor:gammaLimitRepresentation} (iii),
			we obtain
			\begin{align*}
				\lim_{s\rightarrow t_2^-}\mathop\mathrm{ess\,inf}_{\tau\in(s,t_2)}\C E(e(\tau),z(\tau))
				&=\lim_{s\rightarrow t_2^-}\mathop\mathrm{ess\,inf}_{\tau\in(s,t_2)}\C E(e(\tau),z^-(\tau))\\
				&\geq\lim_{s\rightarrow t_2^-}\mathop\mathrm{ess\,inf}_{\tau\in(s,t_2)}\C E(\epsilon(u(\tau)),
					z^-(\tau)\mathds 1_{\frak A_D(\{ z^-(\tau)>0\})})\\
				&\geq \lim_{s\rightarrow t_2^-}\mathop\mathrm{ess\,inf}_{\tau\in(s,t_2)}\frak E (b(\tau),
					z^-(\tau)\mathds 1_{\frak A_D(\{ z^-(\tau)>0\})})\\
				&\geq \frak E (b(t_2),\chi)\\
				&\geq \frak E (b(t_2),z^+(t_2)).
			\end{align*}
			Now we choose $ \frak  e_{t_2}^+ = \frak E (b(t_2),z^+(t_2))$. This leads to
			\begin{align*}
				0\leq \sum_{s\in J_z\cap(t_1,t_2]}\left(\C E^-(s)-\C E^+(s)\right)\leq{}&\lim_{s\rightarrow t_2^-}\mathop\mathrm{ess\,inf}_{\tau\in(s,t_2)}\C E (e(\tau),z(\tau))
					-\frak E(b(t_2),z^+(t_2))\\
				&+\liminf_{s\rightarrow t_2^-}\sum_{\tau\in J_z\cap(t_1,s]}\left(\C  E^-(\tau)-\C E^+(\tau)\right),
			\end{align*}
			where the second '$\leq$' becomes an '$=$' if $t_2\in J_{z}$.
			Consequently, \eqref{eqn:EIextended} becomes
			\begin{align}
				&\frak E(b(t_2),z^+(t_2))
					+\int_{t_1}^{t_2}\int_{F(s)}\left(\alpha|\partial_t^\mathrm{a} z|+\beta|\partial_t^\mathrm{a} z|^2\right)\dxs
					+\sum_{s\in J_z\cap(t_1,t_2]}\left(\C E^-(s)-\C E^+(s)\right)\notag\\
				&\qquad\qquad\qquad\qquad
					\leq \frak e_{t_1}^++\int_{t_1}^{t_2}\int_{F(s)}W_{,e}(e,z):\epsilon(\partial_t b)\dxs.
			\label{eqn:endpointCondition}
			\end{align}
			The energy inequality \eqref{eqn:EI} for $(\widehat e,\widehat u,\widehat z,\widehat F)$
			(taking $\widehat{\frak e}_{t_2}^+=\frak E(\widehat b(t_2),\widehat z^+(t_2))$ into account)
			can be expressed as
			\begin{align}
				&\C E(e(t),z(t))
					+\int_{t_2}^{t}\int_{F(s)}\left(\alpha|\partial_t^\mathrm{a} z|+\beta|\partial_t^\mathrm{a} z|^2\right)\dxs
					+\sum_{s\in J_{z}\cap(t_2,t]}\left(\C E^-(s)-\C E^+(s)\right)\notag\\
				&\qquad\qquad\qquad\qquad
					\leq \frak E(b(t_2),z^+(t_2))+\int_{t_2}^{t}\int_{F(s)}W_{,e}(e,z):\epsilon(\partial_t b)\dxs
			\label{eqn:extendedEnergyInequality}
			\end{align}
			for a.e. $t\in(t_2,t_3)$.
			Adding \eqref{eqn:endpointCondition} and \eqref{eqn:extendedEnergyInequality} shows that the energy estimate
			for $(e,u,z,F)$ also holds for a.e. $t\in(t_2,t_3)$.
			It is now easy to verify that $(e,u,z,F)$ is a weak solution on the time interval $[t_1,t_3]$ according to
			Definition \ref{def:weakSolution}.
			\ep
		\end{proof}\\
		\begin{proof}[Proof of Theorem \ref{theorem:mainExistenceResult}]
			By Zorn's lemma, we deduce the existence of a maximal element
			$R=(\widetilde T,\widetilde e,\widetilde u,\widetilde z,\widetilde F)$ in $\C P$.
			In particular, a maximal element
			satisfies the properties in Theorem \ref{theorem:mainExistenceResult} on the interval $[0,\widetilde T]$.
			We deduce $T=\widetilde T$.
			Otherwise, we get another weak solution $(\widehat e,\widehat u,\widehat z,\widehat F)$
			on $[\widetilde T,\widetilde T+\varepsilon]$ for an $\varepsilon>0$ with initial datum
			$\widetilde z^-(\widetilde T)\mathds 1_{\frak A_D(\widetilde z^-(\widetilde T)>0)}$
			(which is an element of $W^{1,p}(\Omega)$ by Lemma \ref{lemma:truncation})
			as in the proof of Lemma \ref{lemma:Pnonempty} with $e_{\widetilde T}^+=\frak E(b(\widetilde T),z(\widetilde T))$ if $\widetilde T\in J_z$.
			By Lemma \ref{lemma:concatenation}, $(\widetilde e,\widetilde u,\widetilde z,\widetilde F)$ and
			$(\widehat e,\widehat u,\widehat z,\widehat F)$ can be concatenate to a weak solution
			on $[0,\widetilde T+\varepsilon]$ which is a contradiction.
			\ep
		\end{proof}\\
		\begin{proof}[Proof of Remark \ref{theorem:localExistence}]
			Here, let us consider the set $\C P$ given by
			\begin{align*}
				\C P:=\big\{(\widehat T,u,z)\,|\,&0<\widehat T\leq T\text{ and }(u,z)
					\text{ is a weak solution on}\\
				&\text{$[0,\widehat T]$ according to Definition \ref{def:weakSolution}}\big\}
			\end{align*}
			with an ordering $\leq$ as above (except the conditions $e_2|_{[0,\widehat T_1]}=e_1$ and $F_2|_{[0,\widehat T_1]}=F_1$ which are not needed here).
			Proposition \ref{prop:existenceResultSimplified} shows $\C P\neq\emptyset$ by noticing $z\in\C C(\ol{\Omega_T})$ (see
			Theorem \ref{theorem:continuity}) and $0<\eta\leq z^0$.
			The property that every totally ordered subset of $\C P$ has an upper bound can be shown as in
			Lemma \ref{lemma:PupperBound}.
			A maximal element satisfies the claim.
			\ep
		\end{proof}

	\begin{appendix}
		\section{Embedding Theorem}
		The embedding theorem \ref{theorem:continuity} in this appendix is a special version of a more general compactness result in \cite[Corollary 5]{Simon}.
		However, we would like to present a different (short) proof which requires the following generalized version of Poincar\'e's inequality.
		\begin{theorem}[Generalized Poincar\'e inequality {\cite[Section 6.15]{Alt99}}]
		\label{theorem:genPI}
			Let $\Omega\subseteq\R^n$ be a bounded Lipschitz domain and $M\subseteq W^{1,p}(\Omega;\R^m)$ non-empty, convex and closed
			with $1<p<\infty$.
			Furthermore, $M$ satisfies the property
			$$
				u\in M,\;\alpha\geq 0\;\Longrightarrow\;\alpha u\in M.
			$$
			Then the following statements are equivalent:
			\begin{enumerate}
				\renewcommand{\labelenumi}{(\roman{enumi})}
				\item
					There exists a $u_0\in M$ and a constant $C_0>0$ such that for all $\xi\in\R^m$
					$$
						u_0+\xi\in M\;\Longrightarrow\;|\xi|\leq C_0.
					$$
				\item
					There exists a constant $C>0$ such that for all $u\in M$
					$$
						\|u\|_{L^p(\Omega;\R^m)}\leq C\|\nabla u\|_{L^p(\Omega;\R^{m\times n})}.
					$$
			\end{enumerate}
		\end{theorem}
		\begin{theorem}
		\label{theorem:continuity}
			Let $\Omega\subseteq\R^n$ be a bounded Lipschitz domain and $p>n$.
			Then $$L^\infty(0,T;W^{1,p}(\Omega))\cap H^1(0,T;L^2(\Omega))\subseteq\C C(\ol{\Omega_T}).$$
		\end{theorem}
	\begin{proof}
			Let $z\in L^\infty(0,T;W^{1,p}(\Omega))\cap H^1(0,T;L^2(\Omega))$.
			We can choose a representative such that $z\in \C C([0,T];L^2(\Omega))$ and $z(t)\in W^{1,p}(\Omega)$
			for all $t\in[0,T]$.
			By employing the embedding
			$W^{1,p}(\Omega)\subseteq\C C(\ol{\Omega})$ (note that $p>n$), we obtain a representant
			$z:\ol{\Omega_T}\rightarrow\R$ such that
			\begin{align}
			\label{eqn:zContinuity}
				z\in \C C([0,T];L^2(\Omega))\text{ and }z(t)\in \C C(\ol\Omega)\text{ for all }t\in[0,T].
			\end{align}
			Let $(x_m,t_m)\in\ol{\Omega_T}$ be arbitrary with
			$(x_m,t_m)\rightarrow(x,t)$ in $\ol{\Omega_T}$ as $m\rightarrow\infty$. We have
			\begin{align*}
				|z(x,t)-z(x_m,t_m)|\leq \underbrace{|z(x,t)-z(x,t_m)|}_{A_m}+\underbrace{|z(x,t_m)-z(x_m,t_m)|}_{B_m}.
			\end{align*}
			Assume that $A_m\not\rightarrow0$ as $m\rightarrow\infty$.
			Then, there exists a subsequence of $\{A_m\}$ (also denoted by $\{A_m\}$) such that $\lim_{m\rightarrow\infty}A_m>0$.
			Using this subsequence, it holds $z(\cdot,t_m)\rightarrow z(\cdot,t)$ in $L^2(\Omega)$ due to
			\eqref{eqn:zContinuity}.
			We obtain again a subsequence (we omit the additional subscript) such that
			$z(y,t_m)\rightarrow z(y,t)$ as $m\rightarrow\infty$ for a.e. $y\in\Omega$.
			Therefore, we can choose $y_m\rightarrow x$ in $\ol\Omega$ such that $|z(y_m,t)-z(y_m,t_m)|\rightarrow 0$ as
			$m\rightarrow\infty$. It follows
			\begin{align*}
				A_m\leq \underbrace{|z(x,t)-z(y_m,t)|}_{A_m^1}+\underbrace{|z(y_m,t)-z(y_m,t_m)|}_{A_m^2}
					+\underbrace{|z(y_m,t_m)-z(x,t_m)|}_{A_m^3}
			\end{align*}
			The continuity of $z(\cdot,t)$ due to \eqref{eqn:zContinuity}
			implies $A_m^1\rightarrow0$ as $m\rightarrow\infty$.
			$A_m^2$ converges to $0$ by the construction of $\{y_m\}$.
			To treat the term $A_m^3$, we apply the Poincar\'e inequality
			from Theorem \ref{theorem:genPI} with $M:=\{u\in W^{1,p}(B_1(q_0))\,|\,u(q_0)=0\}$ and obtain
			\begin{align}
			\label{eqn:PoincareIneqVariant}
				\|g\|_{L^p(B_1(q_0))}\leq C\|\nabla g\|_{L^p(B_1(q_0))}
			\end{align}
			for all $g\in M$, where $q_0\in\R^n$, and $C>0$ is independent of $g$ and $q_0$.
			Note that, due to $g\in W^{1,p}(B_1(q_0))\subseteq\C C(\ol{B_1(q_0)})$, $g$ is pointwise defined.
			By utilizing \eqref{eqn:PoincareIneqVariant} and using a scaling argument, we gain a $C>0$ such that for all $\varepsilon>0$
			and all $g\in W^{1,p}(B_\varepsilon(q_0))$ with $g(q_0)=0$ it follows
			\begin{align*}
			\begin{split}
				\|g\|_{\C C(\ol{B_\varepsilon(q_0)})}=\|g(\varepsilon\cdot)\|_{\C C(\ol{B_1(q_0)})}
				&\leq C\|g(\varepsilon\cdot)\|_{W^{1,p}(B_1(q_0))}\\
				&\leq C\|\varepsilon \nabla g(\varepsilon\cdot)\|_{L^p(B_1(q_0))}\\
				&=C\varepsilon^\frac{p-n}{p}\|\nabla g\|_{L^p(B_\varepsilon(q_0))}.
			\end{split}
			\end{align*}
			
			By setting $g_m(\cdot):=z(y_m,t_m)-z(\cdot,t_m)$ and $\varepsilon_m:=2|y_m-x|$, we can estimate $A_m^3$ in the following way
			(note that $g_m(y_m)=0$):
			\begin{align*}
				A_m^3\leq\|g_m\|_{\C C(\ol{B_{\varepsilon_m}(y_m)})}
					\leq C{\varepsilon_m}^\frac{p-n}{p}\|\nabla g_m\|_{L^p(B_{\varepsilon_m}(y_m))}.
			\end{align*}
			Since $z\in L^\infty(0,T;W^{1,p}(\Omega))\cap H^1(0,T;L^2(\Omega))$, $\|\nabla g_m\|_{L^p(B_{\varepsilon_m}(y_m))}$ is bounded with respect to $m$.
			In conclusion, $A_m^3\rightarrow 0$ as $m\rightarrow\infty$.
			Hence, we end up with a contradiction. 
			Therefore, $A_m\rightarrow 0$ as $m\rightarrow\infty$.
			
			The convergence $B_m\rightarrow 0$ as $m\rightarrow\infty$ can be shown as for $A_m^3\rightarrow 0$.
			\ep
	\end{proof}
		
	\section{Chain-rule for vector-valued functions of bounded variation}
	\label{section:BV}
		\begin{theorem}[BV-chain rule \cite{MV87}]
		\label{theorem:chainRule}
			Let $I\subseteq\R$ be an interval, $X$ be a real reflexive Banach space, $f\in BV_\mathrm{loc}(I;X)$
			with $\mathrm d f=f'_\mu \mu$ for a non-negative Radon measure $\mu$ on $I$
			and $f'_\mu\in L_\mathrm{loc}^1(I,\mu;X)$.
			Moreover, let $E:X\rightarrow\R$ be continuously Fr\'echet-differentiable.
			Then $E\circ f\in BV_\mathrm{loc}(I;\R)$ and $\mathrm d( E\circ f)$ admits as density relative to $\mu$ the function
			$t\mapsto \langle\theta(t),f'_\mu(t)\rangle$, where $\theta:I\rightarrow X^\star$ is defined as
			$$
				\theta(t):=\int_0^1\mathrm d E((1-r)f(t^-)+rf(t^+))\dr.
			$$
		\end{theorem}
		\begin{corollary}
		\label{cor:chainRule}
			Suppose $f\in SBV(0,T;X)$ and $E:X\rightarrow\R$ is continuously
			Fr\'echet-differentiable.
			Then $E\circ f\in SBV(0,T)$ and for all $0\leq a\leq b\leq T$:
			$$
				\mathrm d (E\circ f)((a,b])=\int_a^b \langle\mathrm d E(f(s)),f'(s)\rangle\ds
				+\sum_{s\in J_f\cap(a,b]}\left(E(f(s^+))-E(f(s^-))\right).
			$$
		\end{corollary}
		\begin{proof}
			We apply Theorem \ref{theorem:chainRule}.
			By assumption, we obtain the decomposition
			$\mathrm df=f'_\mu\mu$ with $\mu=\C L^1+\C H^0\lfloor J_f$ and
			$f'_\mu(t)=f'(t)+f(t^+)-f(t^-)$ for all $t\in(0,T)$.
			Applying Theorem \ref{theorem:chainRule} yields
			\begin{align*}
				\mathrm d (E\circ f)((a,b])&=\int_{(a,b]}\langle\theta(s),f'_\mu(s)\rangle\,\mathrm d\mu(s)\\
				&=\int_{(a,b]}\langle\theta(s),f'(s)\rangle\,\mathrm d\C L^1(s)
					+\sum_{t\in J_f\cap(a,b]}\langle\theta(s),f(s^+)-f(s^-)\rangle
			\end{align*}
			Since $f(s^+)=f(s^-)=f(s)$ for $\C L^1-$a.e. $s\in(a,b]$, the first term on the right hand side becomes
			\begin{align*}
				\int_{(a,b]}\langle\theta(s),f'(s)\rangle\,\mathrm d\C L^1(s)
					&=\int_{(a,b]}\Big\langle\int_0^1\mathrm d E((1-r)f(s^-)+rf(s^+))\dr,f'(s)\Big\rangle\,\mathrm d\C L^1(s)\\
					&=\int_{(a,b]}\langle\mathrm d E(f(s),f'(s)\rangle\ds,
			\end{align*}
			where $\mathrm ds:=\mathrm d\C L^1(s)$.
			Furthermore, by the classical chain rule,
			\begin{align*}
				\sum_{s\in J_f\cap(a,b]}\langle\theta(s),f(s^+)-f(s^-)\rangle
				&=\sum_{s\in J_f\cap(a,b]}\Big\langle\int_0^1\mathrm d E((1-r)f(s^-)+rf(s^+))\dr,f(s^+)-f(s^-)\Big\rangle\\
				&=\sum_{s\in J_f\cap(a,b]}\int_0^1\Big\langle\mathrm d E((1-r)f(s^-)+rf(s^+)),f(s^+)-f(s^-)\Big\rangle\dr\\
				&=\sum_{s\in J_f\cap(a,b]}\int_0^1\frac{\mathrm d}{\mathrm dr}E((1-r)f(s^-)+rf(s^+))\dr\\
				&=\sum_{s\in J_f\cap(a,b]}\Big(E(f(s^+))-E(f(s^-))\Big).
			\end{align*}
			\ep
		\end{proof}
		
	\section{Truncation property for Sobolev functions}
		\begin{lemma}
		\label{lemma:truncation}
			Let $D,\Omega\subseteq\R^n$ be open sets and $p>n$.
			Furthermore, assume that a function $f\in W^{1,p}(\Omega)$ fulfills $f=0$ on $\partial D\setminus\partial\Omega$
			($f$ is here considered as a continuous function due to the embedding $W^{1,p}(\Omega)\hookrightarrow\C C(\ol\Omega)$).
			Then $f\mathds 1_{D}\in W^{1,p}(\Omega)$.
		\end{lemma}
		\begin{proof}
			We can reduce the problem to one space dimension by using the following slicing result from \cite[Proposition 3.105]{Ambrosio00}
			for functions $u\in L^p(\Omega)$:
			\begin{align}
				u\in W^{1,p}(\Omega)\quad\Longleftrightarrow\quad \forall \nu\in\mathbb S^{n-1}:\;&u_x^\nu\in W^{1,p}(\Omega_x^\nu)\text{ for $\C L^{n-1}$-a.e. }
					x\in\Omega_\nu\notag\\
					&\text{ and }\int_{\Omega_\nu}\int_{\Omega_x^\nu}|\nabla u_x^\nu|^p\dt\dy<\infty,
			\label{eqn:slicingProperty}
			\end{align}
			where $\Omega_\nu$ is the orthogonal projection of $\Omega$ to the hyperplane orthogonal to $\nu$ and
			$\Omega_x^\nu:=\{t\in\R\,|\,x+t\nu\in\Omega\}$ as well as $u_x^\nu(t):=u(x+t\nu)$.
			
			Applying this result to $f$, we obtain $f_x^\nu\in W^{1,p}(\Omega_x^\nu)$ for $\C L^{n-1}$-a.e. $x\in\Omega_\nu$ and all $\nu\in\mathbb S^{n-1}$.
			Moreover, slices for the function $g:=f\mathds 1_D$ are given by the equation
			$$
				g_x^\nu=f_x^\nu\mathds 1_{D_x^\nu}.
			$$
			The function $f_x^\nu$ is absolutely continuous. We claim that this is also the case for $g_x^\nu$.
			To proceed, let $\varepsilon>0$ be an arbitrary real.
			Then, we get some constant $\delta>0$ such that
			\begin{align}
				&(a_k,b_k)\text{, $k\in I$, with $a_k\leq b_k$ are finitely many disjoint intervals of }\Omega_x^\nu\text{ with }\sum_{k\in I}|a_k-b_k|<\delta\notag\\
				&\qquad\Longrightarrow\quad\sum_{k\in I}|f_x^\nu(a_k)-f_x^\nu(b_k)|<\varepsilon.
			\label{eqn:absoluteContinuity}
			\end{align}
			The property \eqref{eqn:absoluteContinuity} is also satisfied for $g_x^\nu$.
			Indeed, let $(a_k,b_k)$, $k\in I$, with $a_k\leq b_k$ be finitely many disjoint intervals of $\Omega_x^\nu$ with $\sum_{k\in I}|a_k-b_k|<\delta$.
			We define the values $\widetilde a_k$ and $\widetilde b_k$ in the following way:
			$$
				(\widetilde a_k, \widetilde b_k):=
				\begin{cases}
					(a_k,b_k)&\text{if } a_k,b_k\in D_x^\nu\text{ or }a_k,b_k\not\in D_x^\nu,\\
					(z,b_k)\text{ for an arbitrary fixed }z\in [a_k,b_k]\cap \partial D_x^\nu&\text{if }a_k\not\in D_x^\nu\text{ and } b_k\in D_x^\nu,\\
					(a_k,z)\text{ for an arbitrary fixed }z\in [a_k,b_k]\cap \partial D_x^\nu&\text{if }a_k\in D_x^\nu\text{ and } b_k\not\in D_x^\nu.
				\end{cases}
			$$
			We conclude $\sum_{k\in I} |\widetilde a_k-\widetilde b_k|\leq \sum_{k\in I} |a_k-b_k|\leq\delta$ and therefore
			$\sum_{k\in I}|f_x^\nu(\widetilde a_k)-f_x^\nu(\widetilde b_k)|<\varepsilon$ by \eqref{eqn:absoluteContinuity}.
			Taking
			$$
				\sum_{k\in I}|g_x^\nu(a_k)-g_x^\nu(b_k)|
					=\sum_{k\in I}|g_x^\nu(\widetilde a_k)-g_x^\nu(\widetilde b_k)|\leq \sum_{k\in I}|f_x^\nu(\widetilde a_k)-f_x^\nu(\widetilde b_k)|
			$$
			into account, shows that $g_x^\nu$ is absolutely continuous and we find $g_x^\nu\in W^{1,p}(\Omega_x^\nu)$.
			
			Moreover, $\int_{\Omega_\nu}\int_{\Omega_x^\nu}|\nabla g_x^\nu|^p\dt\dy=\int_{D_\nu}\int_{D_x^\nu}|\nabla f_x^\nu|^p\dt\dy<\infty$.
			Applying \eqref{eqn:slicingProperty} yields $g\in W^{1,p}(\Omega)$.
			\ep
		\end{proof}
	\end{appendix}

\addcontentsline{toc}{chapter}{Bibliography}{\footnotesize{\setlength{\baselineskip}{0.2 \baselineskip}
\bibliography{references}}

\newcommand{\etalchar}[1]{$^{#1}$}
\begin{thebibliography}{GUE{\etalchar{+}}07}

\bibitem[AFP00]{Ambrosio00}
L.~Ambrosio, N.~Fusco, and D.~Pallara.
\newblock {\em Functions of {B}ounded {V}ariation and {F}ree {D}iscontinuity
  {P}roblems}.
\newblock Oxford University Press Inc., New York, 2000.

\bibitem[Alt99]{Alt99}
H.W. Alt.
\newblock {\em Lineare Funktionalanalysis}.
\newblock Springer-Verlag Heidelberg, 1999.

\bibitem[AT90]{AT90}
L.~Ambrosio and V.~M. Tortorelli.
\newblock {Approximation of functional depending on jumps by elliptic
  functional via {G}amma-convergence}.
\newblock {\em CPAM}, 43(8):999--1036, 1990.

\bibitem[Bab11]{Bab11}
J.-F. Babadjian.
\newblock {A quasistatic evolution model for the interaction between fracture
  and damage.}
\newblock {\em Arch. Ration. Mech. Anal.}, 200(3):945--1002, 2011.

\bibitem[BFM08]{BFM08}
B.~Bourdin, G.A. Francfort, and J.-J. Marigo.
\newblock {The variational approach to fracture.}
\newblock {\em J. Elasticity}, 91(1-3):5--148, 2008.

\bibitem[BM14]{BM14}
J.-F. Babadjian and V.~Millot.
\newblock {Unilateral gradient flow of the {A}mbrosio-{T}ortorelli functional
  by minimizing movements}.
\newblock {\em to appear in Annal IHP}, 2014.

\bibitem[BMR09]{BMR09}
G.~Bouchitte, A.~Mielke, and T.~Roub{\'i}{\v c}ek.
\newblock A complete-damage problem at small strains.
\newblock {\em ZAMP Z. Angew. Math. Phys.}, 60:205--236, 2009.

\bibitem[Bra02]{Braides02}
A.~Braides.
\newblock {\em Gamma-convergence for beginners}, volume~1.
\newblock Oxford Lecture Series in Mathematics and its Applications 22. Oxford,
  2002.

\bibitem[BS04]{BS04}
E.~Bonetti and G.~Schimperna.
\newblock Local existence for {F}r\'emond's model of damage in elastic
  materials.
\newblock {\em Contin. Mech. Thermodyn.}, 16(4):319--335, 2004.

\bibitem[BSS05]{BSS05}
E.~Bonetti, G.~Schimperna, and A.~Segatti.
\newblock On a doubly nonlinear model for the evolution of damaging in
  viscoelastic materials.
\newblock {\em J. of Diff. Equations}, 218(1):91--116, 2005.

\bibitem[Car86]{Car86}
A.~Carpinteri.
\newblock {\em {Mechanical damage and crack growth in concrete. Plastic
  collapse to brittle fracture}}.
\newblock {Springer}, Netherlands, 1986.

\bibitem[CFM09]{CFM09}
A.~Chambolle, G.A. Francfort, and J.-J. Marigo.
\newblock {When and how do cracks propagate?}
\newblock {\em J. Mech. Phys. Solids}, 57(9):1614--1622, 2009.

\bibitem[CFM10]{CFM10}
A.~Chambolle, G.A. Francfort, and J.-J. Marigo.
\newblock {Revisiting energy release rates in brittle fracture.}
\newblock {\em J. Nonlinear Sci.}, 20(4):395--424, 2010.

\bibitem[Din66]{Din66}
N.~Dinculeanu.
\newblock {\em {V}ector {M}easures}.
\newblock VEB Deutscher Verlag der Wissenschaften, Berlin, GDR, 1966.

\bibitem[DPO94]{NPO94}
E.A. DeSouzaNeto, D.~Peric, and D.R.J. Owen.
\newblock { A phenomenological three-dimensional rate-independent continuum
  damage model for highly filled polymers: Formulation and computational
  aspects}.
\newblock {\em J. Mech. Phys. Solids}, 42:1533--1550, 1994.

\bibitem[FG06]{FG06}
G.A. Francfort and A.~Garroni.
\newblock {A variational view of partial brittle damage evolution.}
\newblock {\em Arch. Ration. Mech. Anal.}, 182(1):125--152, 2006.

\bibitem[FN96]{FN96}
M.~Fr{\'e}mond and B.~Nedjar.
\newblock Damage, gradient of damage and principle of virtual power.
\newblock {\em Int. J. Solids Structures}, 33(8):1083--1103, 1996.

\bibitem[Fr{\'e}02]{Fre02}
M.~Fr{\'e}mond.
\newblock {\em {Non-smooth thermomechanics.}}
\newblock {Berlin: Springer}, 2002.

\bibitem[Gia05]{Gia05}
A.~Giacomini.
\newblock {Ambrosio-Tortorelli approximation of quasi-static evolution of
  brittle fracture}.
\newblock {\em Calc. Var. Partial Differ. Equ.}, 22(2):129--172, 2005.

\bibitem[GL09]{GL09}
A.~Garroni and C.~Larsen.
\newblock {Threshold-based quasi-static brittle damage evolution.}
\newblock {\em Arch. Ration. Mech. Anal.}, 194(2):585--609, 2009.

\bibitem[GUE{\etalchar{+}}07]{Gee07}
M.G.D. Geers, R.L.J.M. Ubachs, M.~Erinc, M.A. Matin, P.J.G. Schreurs, and W.P.
  Vellinga.
\newblock Multiscale {A}nalysis of {M}icrostructura {E}volution and
  {D}egradation in {S}older {A}lloys.
\newblock {\em Internatilnal Journal for Multiscale Computational Engineering},
  5(2):93--103, 2007.

\bibitem[HK11]{WIAS1520}
C.~Heinemann and C.~Kraus.
\newblock Existence of weak solutions for {C}ahn-{H}illiard systems coupled
  with elasticity and damage.
\newblock {\em Adv. Math. Sci. Appl.}, 21(2):321--359, 2011.

\bibitem[HK13]{WIAS1569}
C.~Heinemann and C.~Kraus.
\newblock Existence results for diffuse interface models describing phase
  separation and damage.
\newblock {\em Eur. J. Appl. Math.}, 24(2):179--211, 2013.

\bibitem[KRZ11]{KRZ11}
D.~Knees, R.~Rossi, and C.~Zanini.
\newblock {\em A vanishing viscosity approach to a rate-independent damage
  model}.
\newblock WIAS preprint no. 1633. WIAS, 2011.

\bibitem[LD05]{LD05}
J.~Lemaitre and R.~Desmorat.
\newblock {\em {E}ngineering {D}amage {M}echanics: {D}uctile, {C}reep,
  {F}atigue and {B}rittle {F}ailures}.
\newblock Springer-Verlag, Berlin, 2005.

\bibitem[LOS10]{LOS10}
C.-J. Larsen, C.~Ortner, and E.~S{\"u}li.
\newblock {Existence of solutions to a regularized model of dynamic fracture}.
\newblock {\em Math. Models Methods Appl. Sci.}, 20(7):1021--1048, 2010.

\bibitem[LT11]{LT11}
G.~Lazzaroni and R.~Toader.
\newblock {A model for crack propagation based on viscous approximation.}
\newblock {\em Math. Models Methods Appl. Sci.}, 21(10):2019--2047, 2011.

\bibitem[Mie95]{Mie95}
C.~Miehe.
\newblock {Discontinuous and continuous damage evolution in Ogden-type
  large-strain elastic materials}.
\newblock {\em Eur. J. Mech.}, 14:697--720, 1995.

\bibitem[Mie11]{Mie11}
A.~Mielke.
\newblock {Complete-damage evolution based on energies and stresses}.
\newblock {\em Discrete Contin. Dyn. Syst., Ser. S}, 4(2):423--439, 2011.

\bibitem[MK00]{MK00}
C.~Miehe and J.~Keck.
\newblock {Superimposed finite elastic-viscoelastic-plastoelastic stress
  response with damage in filled rubbery polymers. Experiments, modelling and
  algorithmic implementation}.
\newblock {\em J. Mech. Phys. Solids}, 48:323--365, 2000.

\bibitem[MR06]{Mielke06}
A.~Mielke and T.~Roub{\'i}{\v c}ek.
\newblock Rate-independent damage processes in nonlinear elasticity.
\newblock {\em Mathematical Models and Methods in Applied Sciences},
  16:177--209, 2006.

\bibitem[MRS08]{MRS08}
A.~Mielke, T.~Roub{\'i}{\v c}ek, and U.~Stefanelli.
\newblock {$\Gamma$-limits and relaxations for rate-independent evolutionary
  problems}.
\newblock {\em Calc. Var. Partial Differ. Equ.}, 31(3):387--416, 2008.

\bibitem[MRZ10]{Mielke10}
A.~Mielke, T.~Roub{\'i}{\v c}ek, and J.~Zeman.
\newblock Complete {D}amage in elastic and viscoelastic media.
\newblock {\em Comput. Methods Appl. Mech. Engrg}, 199:1242--1253, 2010.

\bibitem[MS11]{MS01}
A.~Menzel and P.~Steinmann.
\newblock {A theoretical and computational framework for anisotropic continuum
  damage mechanics at large strains}.
\newblock {\em Int. J. Solids Struct.}, 38:9505--9523, 2011.

\bibitem[MT10]{MT10}
A.~Mielke and M.~Thomas.
\newblock Damage of nonlinearly elastic materials at small strain ---
  {E}xistence and regularity results.
\newblock {\em ZAMM Z. Angew. Math. Mech}, 90:88--112, 2010.

\bibitem[MV87]{MV87}
J.J. Moreau and M.~Valadier.
\newblock A chain rule involving vector functions of bounded variation.
\newblock {\em Journal of Functional Analysis}, 74(2):333--345, 1987.

\bibitem[Neg10]{Neg10}
M.~Negri.
\newblock {From rate-dependent to rate-independent brittle crack propagation.}
\newblock {\em J. Elasticity}, 98(2):159--187, 2010.

\bibitem[RR12]{RR12}
E.~Rocca and R.~Rossi.
\newblock A degenerating {PDE} system for phase transitions and damage.
\newblock {\em arXiv:1205.3578v1}, 2012.

\bibitem[Ser11]{Ser11}
S.~Serfaty.
\newblock {Gamma-convergence of gradient flows on Hilbert and metric spaces and
  applications}.
\newblock {\em DCDS-A}, 31(4):1427--1451, 2011.

\bibitem[Sim86]{Simon}
J.~Simon.
\newblock Compact sets in the space ${L}^p(0,{T};{B})$.
\newblock {\em Annali di Matematica Pura ed Applicata}, 146:65--96, 1986.

\bibitem[VSL11]{VSL11}
G.~Z. Voyiadjis, A.~Shojaei, and G.~Li.
\newblock {A thermodynamic consistent damage and healing model for self healing
  materials.}
\newblock {\em Int. J. Plast.}, 27(7):1025--1044, 2011.

\end{thebibliography}
\bibliographystyle{alpha}}
\end{document}